\documentclass[12pt,a4paper,reqno,oneside]{amsart}
\usepackage{t1enc}
\usepackage{times}
\usepackage{amssymb}
\usepackage[mathscr]{euscript}
\usepackage{graphicx}
\usepackage{hyperref}

\usepackage{comment}

\usepackage{float}
\usepackage{epstopdf}

\usepackage{tikz}

\numberwithin{equation}{section}

\newtheorem{thm}{Theorem}[section]
\newtheorem{prop}[thm]{Proposition}
\newtheorem{lem}[thm]{Lemma}
\newtheorem{cor}[thm]{Corollary}

\theoremstyle{remark}
\newtheorem{rem}[thm]{Remark}

\theoremstyle{definition}
\newtheorem{defn}[thm]{Definition}

\newcommand*{\rom}[1]{\expandafter\@slowromancap\romannumeral #1@} 
\renewcommand{\phi}{\varphi} 
\newcommand{\diff}{\mbox{d}} 
\newcommand{\R}{\mathbb{R}} 

\newcommand{\nooutput}[1]{}

\allowdisplaybreaks

\begin{document}

\title[Strong Solutions of d-Dimensional fBm driven SDE's with Generalized Drift]{Strong Solutions of SDE's with Generalized Drift and Multidimensional Fractional Brownian Initial Noise}

\date{\today}

\author[D. Ba\~{n}os]{David Ba\~{n}os}
\address[David Ba\~{n}os]{\\
Institute of Mathematics of the University of Barcelona, University of Barcelona \\
\\
Gran Via de les Corts Catalanes, 585\\
08007, Barcelona}
\email[]{davidru@math.uio.no}
\author[S. Ortiz-Latorre]{Salvador Ortiz-Latorre}
\address{S. Ortiz-Latorre: Department of Mathematics, University of Oslo, Moltke Moes vei 35, P.O. Box 1053 Blindern, 0316 Oslo, Norway.}
\email{salvadoo@math.uio.no} 
\author[A. Pilipenko]{Andrey Pilipenko}
\address{A. Pilipenko: Institute of Mathematics, National Academy of Sciences of Ukraine, Tereshchenkivska Str. 3, 01601, Kiev, Ukraine.}
\email{pilipenko.ay@yandex.ua} 
\author[F. Proske]{Frank Proske}
\address{F. Proske: CMA, Department of Mathematics, University of Oslo, Moltke Moes vei 35, P.O. Box 1053 Blindern, 0316 Oslo, Norway.}
\email{proske@math.uio.no}



\begin{abstract}
In this paper we prove the existence of strong solutions to a SDE with a
generalized drift driven by a multidimensional fractional Brownian motion
for small Hurst parameters $H<\frac{1}{2}.$ Here the generalized drift is
given as the local time of the unknown solution process, which can be
considered an extension of the concept of a skew Brownian motion to the case
of fractional Brownian motion. Our approach for the construction of strong
solutions is new and relies on techniques from Malliavin calculus combined
with a "local time variational calculus" argument.  

\end{abstract}

\maketitle

\section{Introduction}
Consider the $d-$dimensional stochastic differential equation (SDE) 
\begin{equation}
X_{t}^{x}=x+\alpha L_{t}(X^{x})\cdot \boldsymbol{1}_{d}+B_{t}^{H},0\leq
t\leq T,x\in \mathbb{R}^{d},  \label{SDEfBm}
\end{equation}%
where the driving noise $B_{\cdot }^{H}$ of the this equation is a $d-$%
dimensional fractional Brownian motion, whose components are given by
one-dimensional independent fractional Brownian motions with a Hurst
parameter $H\in (0,1/2),$ and where $\alpha \in \mathbb{R}$ is a constant
and $\boldsymbol{1}_{d}$ is the vector in $\mathbb{R}^{d}$ with entries
given by $1$. Further, $L_{t}(X^{x})$ is the (existing) local time at zero
of $X_{\cdot }^{x},$ which can be formally written as%
\begin{equation*}
L_{t}(X^{x})=\int_{0}^{t}\delta _{0}(X_{s}^{x})ds,
\end{equation*}%
where $\delta _{0}$ denotes the Dirac delta function in $0$.

We also assume that $B_{\cdot }^{H}$ is defined on a complete probability
space $(\Omega ,\mathfrak{A},P).$

We recall here for $d=1$ and Hurst parameter $H\in (0,1)$ that $%
B_{t}^{H},0\leq t\leq T$ is a centered Gaussian process with covariance
structure $R_{H}(t,s)$ given by%
\begin{equation*}
R_{H}(t,s)=E[B_{t}^{H}B_{s}^{H}]=\frac{1}{2}(s^{2H}+t^{2H}-\left\vert
t-s\right\vert ^{2H}).
\end{equation*}%
For $H=\frac{1}{2}$ the fractional Brownian motion $B_{\cdot }^{H}$
coincides with the Brownian motion. Moreover, $B_{\cdot }^{H}$ has a version
with $(H-\varepsilon) $-H\"{o}lder continuous paths for all $\varepsilon \in
(0,H)$ and is the only stationary Gaussian process having the
self-similarity property, that is%
\begin{equation*}
\{B_{\gamma t}^{H}\}_{t\geq 0}=\{\gamma ^{H}B_{t}^{H}\}_{t\geq 0}
\end{equation*}%
in law for all $\gamma >0$. Finally, we mention that for $H\neq \frac{1}{2}$
the fractional Brownian motion is neither a Markov process nor a (weak)
semimartingale. The latter properties however complicate the study of SDE's
driven by $B_{\cdot }^{H}$ and in fact call for the development of new
construction techniques of solutions of such equations beyond the classical
Markovian framework. For further information about the fractional Brownian
motion, the reader may consult e.g. \cite{Nua10} and the references
therein.

In this paper we want to analyze for small Hurst parameters $H\in (0,1/2)$
strong solutions $X_{\cdot }^{x}$ to the SDE (\ref{SDEfBm}), that is solutions
to (\ref{SDEfBm}), which are adapted to a $P$-augmented filtration $\mathcal{F}%
=\{\mathcal{F}_{t}\}_{0\leq t\leq T}$ generated by $B_{\cdot }^{H}$. Let us
mention here that solutions to (\ref{SDEfBm}) can be considered a
generalization of the concept of a \emph{skew Brownian motion} to the case
of a fractional Brownian motion. The skew Brownian motion, which was first
studied in the 1970ties in \cite{ItoMcKean} and \cite{Walsh} and which has
applications to e.g. astrophysics, geophysics or more recently to the
simulation of diffusion processes with discontinuous coefficients (see e.g. \cite{Zhang}, \cite{Lejay03}, \cite{Etore}) ,
is the a solution to the SDE 
\begin{equation}
X_{t}^{x}=x+(2p-1)L_{t}(X^{x})+B_{t},0\leq t\leq T,x\in \mathbb{R},
\label{SkewBm}
\end{equation}%
where $B_{\cdot }$ is a one-dimensional Brownian motion, $L_{t}(X^{x})$ the
local time at zero of $X_{\cdot }^{x}$ and $p$ a parameter, which stands for
the probability of positive excursions of $X_{\cdot }^{x}$.

\bigskip It was shown in \cite{HarrisonShepp} that the SDE (\ref{SkewBm}) has a
unique strong solution if and only if $p\in \lbrack 0,1]$. The approach used
by the latter authors relies on a one-to-one transformation of (\ref{SkewBm}%
) into a SDE without drift and the symmetric It\^{o}-Tanaka formula. An
extension of the latter result to SDE's of the type%
\begin{equation}
dX_{t}=\sigma (X_{t})dB_{t}+\int_{\mathbb{R}}\nu (dx)dL_{t}^{x}(X)
\label{GenSkew}
\end{equation}%
was given in the work \cite{LeGall} under fairly general conditions on the
coefficient $\sigma $ and the measure $\nu ,$ where the author also proves
that strong solutions to (\ref{GenSkew}) can be obtained through a limit of
sequences of solutions to classical It\^{o}-SDE's by using the comparison
theorem.

We remark here that the \emph{Walsh Brownian motion} \cite{Walsh} also provides a natural
extension of the skew Brownian motion, which is a diffusion process on rays
in $\mathbb{R}^{2}$ originating in zero and which exhibits the behaviour of
a Brownian motion on each of those rays. A further generalization of the
latter process is the \emph{spider martingale}, which has been used in the
literature for the study of Brownian filtrations \cite{Yor}.

Other important generalizations of the skew Brownian motion to the
multidimensional case in connection with weak solutions were studied in \cite{Portenko79} and \cite{BC.03}: Using PDE techniques, Portenko in \cite{Portenko79} gives a construction of a unique solution process
associated with an infinitesimal generator with a singular drift
coefficient, which is concentrated on some smooth hypersurface.

On the other hand Bass and Chen in \cite{BC.03} analyze (unique)
weak solutions of equations of the form%
\begin{equation}
dX_{t}=dA_{t}+dB_{t},  \label{Bass}
\end{equation}%
where $B_{\cdot }$ is a $d-$dimensional Brownian motion and $A_{t}$ a process %
, which is obtained from limits of the form%
\begin{equation*}
\lim_{n\longrightarrow \infty }\int_{0}^{t}b_{n}(X_{s})ds
\end{equation*}%
in the sense of probability uniformly over time $t$ for functions $b_{n}:%
\mathbb{R}^{d}\longrightarrow \mathbb{R}^{d}$. Here the $i$th components of $%
A_{t}$ are bounded variation processes, which correspond to signed measures
in the Kato class $K_{d-1}$. The method of the authors for the construction
of unique weak solutions of such equations is based on the construction of a
certain resolvent family on the space $C_{b}(\mathbb{R}^{d})$ in connection
with the properties of the Kato class $K_{d-1}$.

In this context we also mention the paper \cite{FRW} on SDE's with distributional drift
coefficients. As for a general overview of various construction techniques
with respect to the skew Brownian motion and related processes based e.g. on the
theory of Dirichlet forms or martingale problems, the reader is referred to \cite{Lejay06}. See also the book \cite{Pilipenko}.

\bigskip 

The objective of this paper is the construction of strong solutions to the
multidimensional SDE (\ref{SDEfBm}) with fractional Brownian noise initial data
for small Hurst parameters $H<\frac{1}{2},$ where the generalized drift is
given by the local time of the unknown process. Note that in contrast to \cite{HarrisonShepp} in the case of a skew Brownian motion we obtain in this
article the existence of strong solutions to (\ref{SDEfBm}) for \emph{all} parameters $\alpha
\in \mathbb{R}.$   

Since the fractional Brownian motion is neither a Markov process nor a
semimartingale, if $H\neq \frac{1}{2}$, the methods of the above mentioned
authors cannot be (directly) used for the construction of strong solutions
in our setting. In fact, our construction technique considerably differs
from those in the literature in the Wiener case. More specifically, we
approximate the Dirac delta function in zero by means of functions $\varphi
_{\varepsilon }$ for $\varepsilon \searrow 0$ given by%
\begin{equation*}
\varphi _{\varepsilon }(x)=\varepsilon ^{-\frac{d}{2}}\varphi (\varepsilon
^{-\frac{1}{2}}x),x\in \mathbb{R}^{d}
\end{equation*}%
where $\varphi $ is e.g. the $d-$dimensional standard Gaussian density. Then
we prove that the sequence of strong solutions $X_{t}^{n}$ to the SDE's%
\begin{equation*}
X_{t}^{n}=x+\int_{0}^{t}\alpha\varphi _{1/n}(X_{s}^{n})\cdot \boldsymbol{1}_{d}ds+B_{t}^{H}
\end{equation*}%
converges in $L^{2}(\Omega )$, strongly to a solution to (\ref{SDEfBm}) for $%
n\longrightarrow \infty $. In showing this we employ a compactness criterion
for sets in $L^{2}(\Omega )$ based on\ Malliavin calculus combined with a
"local time variational calculus" argument. See \cite{BNP.16} for the
existence of strong solutions of SDE's driven by $B_{\cdot }^{H},$ $H<\frac{1%
}{2}$, when e.g. the drift coefficients $b$ belong to $L^{1}(\mathbb{R}%
^{d})\cap L^{\infty }(\mathbb{R}^{d})$ or see \cite{MBP10} in the Wiener case. We
also refer to a series of other papers in the Wiener and L\'{e}vy process
case and in the Hilbert space setting based on that approach: \cite{MMNPZ10}, \cite{HaaPros.14}, \cite{MNP14}, \cite{FNP.13}, \cite{BDMP1}, \cite{BDMP2}. 

Here we also want to point out a recent work of Catellier, Gubinelli \cite{CG},
which came to our attention, after having finalized our article. In their
striking paper, which extends the results of Davie \cite{Davie} to the case of a
fractional Brownian noise, the authors study the problem, which fractional
Brownian paths actually regularize solutions to SDE's of the form%
\begin{equation*}
dX_{t}^{x}=b(X_{t}^{x})dt+dB_{t}^{H},X_{0}^{x}=x\in \mathbb{R}^{d}
\end{equation*}
for all $H\in (0,1)$. The (unique) solutions constructed in \cite{CG} are \emph{%
path by path} with respect to time-dependent vector fields $b$ in the Besov-H%
\"{o}lder space $B_{\infty ,\infty }^{\alpha +1},$ $\alpha \in \mathbb{R}$
and in the case of distributional vector fields solutions to the SDE's,
where the drift term is given by a non-linear Young type of integral based
on an averaging operator. In proving existence and uniqueness results the
authors use the Leray-Schauder-Tychonoff fixed point theorem and a
comparison principle in connection with an average translation operator.
Further, Lipschitz-regularity of the flow $(x\longmapsto X_{t}^{x})$ under
certain conditions is shown.

We remark that our techniques are very different from those developed by
Catellier, Gubinelli \cite{CG}, which seem not to work in the case of the SDE (\ref{SDEfBm})
(private communication with one of the authors in \cite{CG}). Further, their methods fail in
the case of SDE's with vector fields $b$, which are merely bounded and
measurable.

\bigskip Finally, we mention that the construction technique in this article
may be also used for showing strong solutions of SDE's with respect to
generalized drifts in the sense of (\ref{Bass}) based on Kato classes. The
existence of strong solutions of such equations in the Wiener case is to the
best of our knowledge still an open problem.\ See the work of Bass, Chen \cite{BC.03}.

\bigskip 

Our paper is organized as follows: In\ Section 2 we introduce the framework
of our paper and recall in this context some basic facts from fractional
calculus and Malliavin calculus for (fractional) Brownian noise. Further, in
Section 3 we discuss an integration by parts formula based on a local time
on a simplex, which we want to employ in connection with a compactness
criterion from Malliavin calculus in Section 5. Section 4 is devoted to the
study of the local time of the fractional Brownian motion and its
properties. Finally, in Section 5 we prove the existence of a strong
solution to (\ref{SDEfBm}) by using the results of the previous sections.   

\bigskip

\section{Framework}\label{Frame}
In this section we pass in review some theory on fractional calculus, Malliavin calculus for fractional Brownian noise and occupation measures which will be progressively used throughout the article. The reader might consult \cite{Mall97}, \cite{Mall78} or \cite{DOP08} for a general theory on Malliavin calculus for Brownian motion and \cite[Chapter 5]{Nua10} for fractional Brownian motion. For theory on occupation measures we refer to \cite{geman.horo.80} or \cite{Kar98}.

\subsection{Fractional calculus}\label{fraccal}
We start up here with some basic definitions and properties of fractional derivatives and integrals. For more information see \cite{samko.et.al.93} and \cite{lizorkin.01}.

Let $a,b\in \R$ with $a<b$. Let $f\in L^p([a,b])$ with $p\geq 1$ and $\alpha>0$. Introduce the \emph{left-} and \emph{right-sided Riemann-Liouville fractional integrals} by
$$I_{a^+}^\alpha f(x) = \frac{1}{\Gamma (\alpha)} \int_a^x (x-y)^{\alpha-1}f(y)dy$$
and
$$I_{b^-}^\alpha f(x) = \frac{1}{\Gamma (\alpha)} \int_x^b (y-x)^{\alpha-1}f(y)dy$$
for almost all $x\in [a,b]$ where $\Gamma$ is the Gamma function.

Further, for a given integer $p\geq 1$, let $I_{a^+}^{\alpha} (L^p)$ (resp. $I_{b^-}^{\alpha} (L^p)$) be the image of $L^p([a,b])$ of the operator $I_{a^+}^\alpha$ (resp. $I_{b^-}^\alpha$). If $f\in I_{a^+}^{\alpha} (L^p)$ (resp. $f\in I_{b^-}^{\alpha} (L^p)$) and $0<\alpha<1$ then define the \emph{left-} and \emph{right-sided Riemann-Liouville fractional derivatives} by
$$D_{a^+}^{\alpha} f(x)= \frac{1}{\Gamma (1-\alpha)} \frac{\diff}{\diff x} \int_a^x \frac{f(y)}{(x-y)^{\alpha}}dy$$
and
$$D_{b^-}^{\alpha} f(x)= \frac{1}{\Gamma (1-\alpha)} \frac{\diff}{\diff x} \int_x^b \frac{f(y)}{(y-x)^{\alpha}}dy.$$

The left- and right-sided derivatives of $f$ defined as above can be represented as follows by
$$D_{a^+}^{\alpha} f(x)= \frac{1}{\Gamma (1-\alpha)} \left(\frac{f(x)}{(x-a)^\alpha}+\alpha\int_a^x \frac{f(x)-f(y)}{(x-y)^{\alpha+1}}dy\right)$$
and
$$D_{b^-}^{\alpha} f(x)= \frac{1}{\Gamma (1-\alpha)} \left(\frac{f(x)}{(b-x)^\alpha}+\alpha\int_x^b \frac{f(x)-f(y)}{(y-x)^{\alpha+1}}dy\right).$$

Finally, we see by construction that the following relations are valid
$$I_{a^+}^\alpha (D_{a^+}^{\alpha} f) = f$$
for all $f\in I_{a^+}^{\alpha} (L^p)$ and
$$D_{a^+}^{\alpha}(I_{a^+}^\alpha  f) = f$$
for all $f\in L^p([a,b])$ and similarly for $I_{b^-}^{\alpha}$ and $D_{b^-}^{\alpha}$.

\subsection{Shuffles}\label{shuffles}
Let $m$ and $n$ be integers. We denote by $S(m,n)$ the set of \emph{shuffle permutations}, i.e. the set of permutations $\sigma: \{1, \dots, m+n\} \rightarrow \{1, \dots, m+n\}$ such that $\sigma(1) < \dots < \sigma(m)$ and $\sigma(m+1) < \dots < \sigma(m+n)$.

The $m$-dimensional simplex is defined as
$$
\Delta_{\theta,t}^m := \{(s_m,\dots,s_1)\in [0,T]^m : \, \theta<s_m<\cdots < s_1<t\}.
$$
The product of two simplices then is given by the following union
$$
\Delta_{\theta,t}^m \times \Delta_{\theta,t}^n = \mbox{\footnotesize $\bigcup_{\sigma \in S(m,n)} \{(w_{m+n},\dots,w_1)\in [0,T]^{m+n} : \, \theta< w_{\sigma(m+n)} <\cdots < w_{\sigma(1)} <t\} \cup \mathcal{N}$ \normalsize},
$$
where the set $\mathcal{N}$ has null Lebesgue measure. Thus, if $f_i:[0,T] \rightarrow \R$, $i=1,\dots,m+n$ are integrable functions we obtain that
\begin{align} \label{shuffleIntegral}
\int_{\Delta_{\theta,t}^m} \prod_{j=1}^m f_j(s_j) ds_m \dots ds_1 & \int_{\Delta_{\theta,t}^n} \prod_{j=m+1}^{m+n} f_j(s_j) ds_{m+n} \dots ds_{m+1}  \notag \\
&= \sum_{\sigma\in S(m,n)} \int_{\Delta_{\theta,t}^{m+n}} \prod_{j=1}^{m+n} f_{\sigma(j)} (w_j) dw_{m+n}\cdots dw_1. 
\end{align}

We hereby give a slight generalization of the above lemma, whose proof can be also found in \cite{BNP.16}. This lemma will be used in Section 5. The reader may skip this lemma at first reading.

\begin{lem}\label{partialshuffle}
Let $n,p$ and $k$ be integers, $k \leq n$. Assume we have integrable functions $f_j : [0,T] \rightarrow \R$, $j = 1, \dots, n$ and $g_i : [0,T] \rightarrow \R$, $i = 1, \dots, p$. We may then write
\begin{align*}
\int_{\Delta_{\theta,t}^n} f_1(s_1) \dots f_k(s_k) \int_{\Delta_{\theta, s_k}^p} g_1(r_1) \dots g_p(r_p) dr_p \dots dr_1 f_{k+1}(s_{k+1}) \dots f_n(s_n) ds_n \dots ds_1 \\
= \sum_{\sigma \in A_{n,p}} \int_{\Delta_{\theta,t}^{n+p}} h^{\sigma}_1(w_1) \dots h^{\sigma}_{n+p}(w_{n+p}) dw_{n+p} \dots dw_1,
\end{align*}
where $h^{\sigma}_l \in \{ f_j, g_i : 1 \leq j \leq n, 1 \leq i \leq p\}$. Here $A_{n,p}$ is a subset of permutations of $\{1, \dots, n+p\}$ such that $\# A_{n,p} \leq C^{n+p}$ for a constant $C \geq 1$, and we use the definition $s_0 = \theta$.
\end{lem}

\begin{proof}
The proof of the result is given by induction on $n$. For $n=1$ and $k=0$ the result is trivial. For $k=1$ we have
\begin{align*}
\int_{\theta}^t f_1(s_1) \int_{\Delta_{\theta,s_1}^p} g_1(r_1) \dots g_p(r_p) & dr_p \dots dr_1 ds_1 \\
 &  = \int_{\Delta_{\theta,t}^{p+1}} f_1(w_1) g_1(w_2) \dots g_p(w_{p+1})  dw_{p+1} \dots dw_1,
\end{align*}
where we have put $w_1 =s_1, w_2 =  r_1, \dots, w_{p+1} = r_p$.

Assume the result holds for $n$ and let us show that this implies that the result is true for $n+1$. Either $k=0,1$ or $2 \leq k \leq n+1$. For $k=0$ the result is trivial. For $k=1$ we have
\begin{align*}
\int_{\Delta_{\theta,t}^{n+1}} & f_1(s_1) \int_{\Delta_{\theta,s_1}^p} g_1(r_1) \dots g_p(r_p) dr_p \dots dr_1 f_2(s_2) \dots  f_{n+1}(s_{n+1}) ds_{n+1} \dots ds_1 \\
&= \int_{\theta}^t f_1(s_1) \left( \int_{\Delta_{\theta,s_1}^n} \int_{\Delta_{\theta,s_1}^p} g_1(r_1) \dots g_p(r_p) dr_p \dots dr_1 f_2(s_2) \dots  f_{n+1}(s_{n+1}) ds_{n+1} \dots ds_2 \right) ds_1 .
\end{align*}
The result follows from (\ref{shuffleIntegral}) coupled with $ \# S(n,p) = \frac{(n+p)!}{n! p!} \leq C^{n+p} \leq C^{(n+1) + p}$. For $k \geq 2$ we have from the induction hypothesis
\begin{align*}
\int_{\Delta_{\theta,t}^{n+1}}   f_1(s_1) \dots f_k(s_k) \int_{\Delta_{\theta, s_k}^p}  g_1(r_1) \dots g_p(r_p) & dr_p \dots dr_1 f_{k+1}(s_{k+1}) \dots f_{n+1}(s_{n+1}) ds_{n+1} \dots ds_1 \\
 = \int_{\theta}^t f_1(s_1)   \int_{\Delta_{\theta,s_1}^n} f_2(s_2) \dots f_k(s_k) & \int_{\Delta_{\theta, s_k}^p} g_1(r_1) \dots g_p(r_p) dr_p \dots dr_1  \\
&  \times  f_{k+1}(s_{k+1}) \dots f_{n+1}(s_{n+1}) ds_{n+1} \dots ds_2  ds_1 \\
 =   \sum_{\sigma \in A_{n,p}} \int_{\theta}^t f_1(s_1)   \int_{\Delta_{\theta,s_1}^{n+p}} & h^{\sigma}_1(w_1) \dots h^{\sigma}_{n+p}(w_{n+p}) dw_{n+p} \dots dw_1 ds_1\\
= \sum_{\tilde{\sigma} \in A_{n+1,p}} \int_{\Delta_{\theta,t}^{n+1+p}} & h^{\tilde{\sigma}}_1(w_1) \dots \tilde{h}^{\tilde{\sigma}}_{w_{n+1+p}} dw_1 \dots dw_{n+1+p},
\end{align*}
where $A_{n+1,p}$ is the set of permutations $\tilde{\sigma}$ of $\{1, \dots, n+1+p\}$ such that $\tilde{\sigma}(1) = 1$ and $\tilde{\sigma}(j+1) = \sigma(j)$, $j=1, \dots, n+p$ for some $\sigma \in A_{n,p}$ .

\end{proof}

\begin{rem}
We remark that the set $A_{n,p}$ in the above Lemma also depends on $k$ but we shall not make use of this fact.
\end{rem}

\subsection{Fractional Brownian motion}
Denote by $B^H = \{B_t^H, t\in [0,T]\}$ a $d$-dimensional \emph{fractional Brownian motion} with Hurst parameter $H\in (0,1/2)$. So $B^H$ is a centered Gaussian process with covariance structure
$$(R_H(t,s))_{i,j}:= E[B_t^{H,(i)} B_s^{H,(j)}]=\delta _{ij}\frac{1}{2}\left(t^{2H} + s^{2H} - |t-s|^{2H} \right), \quad i,j=1,\dots,d,$$
where $\delta _{ij}$ is one, if $i=j$, or zero else.
Observe that $E[|B_t^H - B_s^H|^2]= d|t-s|^{2H}$ and hence $B^H$ has stationary increments and H\"{o}lder continuous trajectories of index $H-\varepsilon$ for all $\varepsilon\in(0,H)$. Observe that the increments of $B^H$, $H\in (0,1/2)$ are not independent. As a matter of fact, this process does not satisfy the Makov property, either. Another obstacle one is faced with is that $B^H$ is not a semimartingale, see e.g. \cite[Proposition 5.1.1]{Nua10}.

We give an abridged survey on how to construct fractional Brownian motion via an isometry. We will do it in one dimension inasmuch as we will treat the multidimensional case componentwise. See \cite{Nua10} for further details.

Let $\mathcal{E}$ be the set of step functions on $[0,T]$ and let $\mathcal{H}$ be the Hilbert space given by the closure of $\mathcal{E}$ with respect to the inner product 
$$\langle 1_{[0,t]} , 1_{[0,s]}\rangle_{\mathcal{H}} = R_H(t,s).$$
The mapping $1_{[0,t]} \mapsto B_t$ has an extension to an isometry between $\mathcal{H}$ and the Gaussian subspace of $L^2(\Omega)$ associated with $B^H$. We denote the isometry by $\varphi \mapsto B^H(\varphi)$. Let us recall the following result (see \cite[Proposition 5.1.3]{Nua10} ) which gives an integral representation of $R_H(t,s)$ when $H<1/2$:

\begin{prop}
Let $H<1/2$. The kernel
$$K_H(t,s)= c_H \left[\left( \frac{t}{s}\right)^{H- \frac{1}{2}} (t-s)^{H- \frac{1}{2}} + \left( \frac{1}{2}-H\right) s^{\frac{1}{2}-H} \int_s^t u^{H-\frac{3}{2}} (u-s)^{H-\frac{1}{2}} du\right],$$
where $c_H = \sqrt{\frac{2H}{(1-2H) \beta(1-2H , H+1/2)}}$ being $\beta$ the Beta function, satisfies
\begin{align}\label{RH}
R_H(t,s) = \int_0^{t\wedge s} K_H(t,u)K_H(s,u)du.
\end{align}
\end{prop}

The kernel $K_H$ also has the following representation by means of fractional derivatives
$$K_H(t,s) = c_H \Gamma \left( H+\frac{1}{2}\right) s^{\frac{1}{2}-H} \left( D_{t^-}^{\frac{1}{2}-H} u^{H-\frac{1}{2}}\right)(s).$$

Consider now the linear operator $K_H^{\ast}: \mathcal{E} \rightarrow L^2([0,T])$ defined by
$$(K_H^{\ast} \varphi)(s) = K_H(T,s)\varphi(s) + \int_s^T (\varphi(t)-\varphi(s)) \frac{\partial K_H}{\partial t}(t,s)dt$$
for every $\varphi \in \mathcal{E}$. We see that $(K_H^{\ast} 1_{[0,t]})(s) = K_H(t,s)1_{[0,t]}(s)$, then from this fact and \eqref{RH} one can conclude that $K_H^{\ast}$ is an isometry between $\mathcal{E}$ and $L^2([0,T])$ which extends to the Hilbert space $\mathcal{H}$. See e.g. \cite{decreu.ustunel.98} and \cite{alos.mazet.nualart.01} and the references therein.

For a given $\varphi\in \mathcal{H}$ one proves that $K_H^{\ast}$can be represented in terms of fractional derivatives in the following ways
$$(K_H^{\ast} \varphi)(s) = c_H \Gamma\left( H+\frac{1}{2}\right) s^{\frac{1}{2}-H} \left(D_{T^-}^{\frac{1}{2}-H} u^{H-\frac{1}{2}}\varphi(u)\right) (s)$$
and
\begin{align*}
(K_H^{\ast} \varphi)(s) =& \, c_H \Gamma\left( H+\frac{1}{2}\right)\left(D_{T^-}^{\frac{1}{2}-H} \varphi(s)\right) (s)\\
&+ c_H \left( \frac{1}{2}-H\right)\int_s^T  \varphi(t) (t-s)^{H-\frac{3}{2}} \left(1- \left(\frac{t}{s}\right)^{H-\frac{1}{2}}\right)dt.
\end{align*}

One finds that $\mathcal{H} = I_{T^-}^{\frac{1}{2}-H}(L^2)$ (see \cite{decreu.ustunel.98} and \cite[Proposition 6]{alos.mazet.nualart.01}).

Using the fact that $K_H^{\ast}$ is an isometry from $\mathcal{H}$ into $L^2([0,T])$ the $d$-dimensional process $W=\{W_t, t\in [0,T]\}$ defined by
\begin{align}\label{WBH}
W_t := B^H((K_H^{\ast})^{-1}(1_{[0,t]}))
\end{align}
is a Wiener process and the process $B^H$ can be represented as follows
\begin{align}\label{BHW}
B_t^H = \int_0^t K_H(t,s) dW_s,
\end{align}
see \cite{alos.mazet.nualart.01}.

We also need to introduce the concept of fractional Brownian motion associated
with a filtration.
\begin{defn}
Let $\mathcal{G}=\left\{ \mathcal{G}_{t}\right\} _{t\in\left[0,T\right]}$
be a right-continuous increasing family of $\sigma$-algebras on $\left(\Omega,\mathcal{F},P\right)$
such that $\mathcal{G}_{0}$ contains the null sets. A fractional
Brownian motion $B^{H}$ is called a $\mathcal{G}$-fractional Brownian
motion if the process $W$ defined by \eqref{WBH} is a $\mathcal{G}$-Brownian
motion.
\end{defn}

In what follows, we will denote by $W$ a standard Wiener process on a given probability space $(\Omega, \mathfrak{A}, P)$ equipped with the natural filtration $\mathcal{F}=\{\mathcal{F}_t\}_{t\in [0,T]}$ which is generated by $W$ and augmented by all $P$-null sets, we shall denote by $B:=B^H$ the fractional Brownian motion with Hurst parameter $H\in (0,1/2)$ given by the representation \eqref{BHW}.

In this paper, we want to make use of a version of Girsanov's theorem for fractional Brownian motion which is due to \cite[Theorem 4.9]{decreu.ustunel.98}. Here we recall the version given in \cite[Theorem 2]{nualart.ouknine.02}. However, we first need the definition of an isomorphism $K_H$ from $L^2([0,T])$ onto $I_{0+}^{H+\frac{1}{2}}(L^2)$ associated with the kernel $K_H(t,s)$ in terms of the fractional integrals as follows, see \cite[Theorem 2.1]{decreu.ustunel.98}
$$(K_H \varphi)(s) = I_{0^+}^{2H} s^{\frac{1}{2}-H} I_{0^+}^{\frac{1}{2}-H}s^{H-\frac{1}{2}}  \varphi, \quad \varphi \in L^2([0,T]).$$

It follows from this and the properties of the Riemann-Liouville fractional integrals and derivatives that the inverse of $K_H$ takes the form
$$(K_H^{-1} \varphi)(s) = s^{\frac{1}{2}-H} D_{0^+}^{\frac{1}{2}-H} s^{H-\frac{1}{2}} D_{0^+}^{2H} \varphi(s), \quad \varphi \in I_{0+}^{H+\frac{1}{2}}(L^2).$$

The latter implies that if $\varphi$ is absolutely continuous, see \cite{nualart.ouknine.02}, one has
\begin{align}\label{inverseKH}
(K_H^{-1} \varphi)(s) = s^{H-\frac{1}{2}} I_{0^+}^{\frac{1}{2}-H} s^{\frac{1}{2}-H}\varphi'(s).
\end{align}

\begin{thm}[Girsanov's theorem for fBm]\label{girsanov}
Let $u=\{u_t, t\in [0,T]\}$ be an $\mathcal{F}$-adapted process with integrable trajectories  and set
$\widetilde{B}_t^H = B_t^H + \int_0^t u_s ds, \quad t\in [0,T].$
Assume that
\begin{itemize}
\item[(i)] $\int_0^{\cdot} u_s ds \in I_{0+}^{H+\frac{1}{2}} (L^2 ([0,T])$, $P$-a.s.

\item[(ii)] $E[\xi_T]=1$ where
$$\xi_T := \exp\left\{-\int_0^T K_H^{-1}\left( \int_0^{\cdot} u_r dr\right)(s)dW_s - \frac{1}{2} \int_0^T K_H^{-1} \left( \int_0^{\cdot} u_r dr \right)^2(s)ds \right\}.$$
\end{itemize}
Then the shifted process $\widetilde{B}^H$ is an $\mathcal{F}$-fractional Brownian motion with Hurst parameter $H$ under the new probability $\widetilde{P}$ defined by $\frac{d\widetilde{P}}{dP}=\xi_T$.
\end{thm}

\begin{rem}
As for the multidimensional case, define
$$(K_H \varphi)(s):= ( (K_H \varphi^{(1)} )(s), \dots, (K_H  \varphi^{(d)})(s))^{\ast}, \quad \varphi \in L^2([0,T];\R^d),$$
where $\ast$ denotes transposition. Similarly for $K_H^{-1}$ and $K_H^{\ast}$.
\end{rem}

In this paper, we will also employ a crucial property of the fractional Brownian motion which was shown by \cite{pitt.78}  for general Gaussian vector fields. The latter property will be a helpful substitute for the lack of independent increments of the underlying noise.

Let $m\in \mathbb{N}$ and $0=:t_0<t_1<\cdots <t_m<T$. Then for all $\xi_1,\dots, \xi_m\in \R^d$ there exists a positive finite constant $C>0$ (depending on $m$) such that
\begin{align}\label{SLND}
\mathrm{Var}\left[ \sum_{j=1}^m \langle\xi_j, B_{t_j}-B_{t_{j-1}}\rangle_{\R^d}\right] \geq C \sum_{j=1}^m |\xi_j|^2 \mathrm{Var}\left[B_{t_j}-B_{t_{j-1}}\right].
\end{align}

The above property is referred to in literature as local non-determinism property of the fractional Brownian motion. The reader may consult \cite{pitt.78} or \cite{xiao.11} for more information on this property. A stronger version of local non-determinism is also satisfied by the fractional Brownian motion. There exists a constant $K>0$, depending only on $H$ and $T$, such
that for any $t\in\left[0,T\right],0<r<t$ and for $i=1,\ldots,d,$
\begin{align}\label{eq:SLND}
\mathrm{Var}\left[B_{t}^{H,i}|\left\{ B_{s}^{H,i}:\left|t-s\right|\geq r\right\} \right]\geq Kr^{2H}.
\end{align}

\section{An integration by parts formula}\label{Section3}

In this section we recall an integration by parts formula, which is essentially based on the local time of the Gaussian process $B^H$. The whole content as well as the proofs can be found in \cite{BNP.16}.

Let $m$ be an integer and let $f :[0,T]^m \times (\R^d)^m \rightarrow \R$ be a function of the form
\begin{align}\label{f}
f(s,z)= \prod_{j=1}^m f_j(s_j,z_j),\quad s = (s_1, \dots,s_m) \in [0,T]^m, \quad z = (z_1, \dots, z_m) \in (\R^d)^m,
\end{align}
where $f_j:[0,T]\times \R^d \rightarrow \R$, $j=1,\dots,m$ are smooth functions with compact support. Further, let $\varkappa:[0,T]^m\rightarrow \R$ be a function of the form
\begin{align}\label{kappa}
\varkappa(s)= \prod_{j=1}^m \varkappa_j(s_j), \quad s\in [0,T]^m,
\end{align}
where $\varkappa_j : [0,T] \rightarrow \R$, $j=1,\dots, m$ are integrable functions.

Next, denote by $\alpha_j$ a multiindex and $D^{\alpha_j}$ its corresponding differential operator. For $\alpha = (\alpha_1, \dots, \alpha_m)$ considered an element of $\mathbb{N}_0^{d\times m}$ so that $|\alpha|:= \sum_{j=1}^m \sum_{l=1}^d \alpha_{j}^{(l)}$, we write
$$
D^{\alpha}f(s,z) = \prod_{j=1}^m D^{\alpha_j} f_j(s_j,z_j).
$$

In this section we aim at deriving an integration by parts formula of the form
\begin{equation} \label{ibp}
\int_{\Delta_{\theta,t}^m} D^{\alpha}f(s,B_s) ds = \int_{(\R^d)^m} \Lambda^{f}_{\alpha} (\theta,t,z)dz ,
\end{equation}
for a suitable random field $\Lambda^f_{\alpha}$, where $\Delta_{\theta,t}^m$ is the $m$-dimensional simplex as defined in Section 2.2 and $B_{s}=(B_{s_{1}},...,B_{s_{m}})$ on that simplex. More specifically, we have that
\begin{equation} \label{LambdaDef}
\Lambda^f_{\alpha}(\theta,t ,z) = (2 \pi)^{-dm} \int_{(\R^d)^m} \int_{\Delta_{\theta,t}^m} \prod_{j=1}^m f_j(s_j,z_j) (-i u_j)^{\alpha_j} \exp \{ -i \langle u_j, B_{s_j} - z_j \rangle\}ds du .
\end{equation}

Let us start by \emph{defining} $\Lambda^f_{\alpha}(\theta,t,z)$ as above and show that it is a well-defined element of $L^2(\Omega)$.

To this end, we need the following notation: Given $(s,z) = (s_1, \dots, s_m ,z_1 \dots, z_m)  \in [0,T]^m \times (\R^d)^m$ and a shuffle $\sigma \in S(m,m)$ we write
$$
f_{\sigma}(s,z) := \prod_{j=1}^{2m} f_{[\sigma(j)]}(s_j, z_{[\sigma(j)]})
$$
and
$$\varkappa_{\sigma} (s) := \prod_{j=1}^{2m} \varkappa_{[\sigma(j)]}(s_j),$$
where $ [j] $ is equal to $j$ if $1 \leq j \leq m$ and $j-m$ if $m+1 \leq j \leq 2m $.

For integers $k \geq 0$ let us define the expressions
\begin{eqnarray*}
&&\Psi _{k}^{f}(\theta ,t,z) \\
&:&=\prod_{l=1}^{d}\sqrt{(2\left\vert \alpha ^{(l)}\right\vert )!}%
\sum_{\sigma \in S(m,m)}\int_{\Delta _{0,t}^{2m}}\left\vert f_{\sigma
}(s,z)\right\vert \prod_{j=1}^{2m}\frac{1}{\left\vert
s_{j}-s_{j-1}\right\vert ^{H(d+2\sum_{l=1}^{d}\alpha _{\lbrack \sigma
(j)]}^{(1)})}}ds_{1}...ds_{2m}
\end{eqnarray*}

respectively,
\begin{eqnarray*}
&&\Psi _{k}^{\varkappa }(\theta ,t) \\
&:&=\prod_{l=1}^{d}\sqrt{(2\left\vert \alpha ^{(l)}\right\vert )!}%
\sum_{\sigma \in S(m,m)}\int_{\Delta _{0,t}^{2m}}\left\vert \varkappa
_{\sigma }(s)\right\vert \prod_{j=1}^{2m}\frac{1}{\left\vert
s_{j}-s_{j-1}\right\vert ^{H(d+2\sum_{l=1}^{d}\alpha _{\lbrack \sigma
(j)]}^{(1)})}}ds_{1}...ds_{2m}.
\end{eqnarray*}

\begin{thm}\label{mainthmlocaltime}
Suppose that $\Psi _{k}^{f}(\theta ,t,z),\Psi _{k}^{\varkappa
}(\theta ,t)<\infty $. Then, defining $\Lambda _{\alpha }^{f}(\theta ,t,z)$
as in \eqref{LambdaDef} gives a random variable in $L^{2}(\Omega )$ and there exists a
universal constant $C=C(T,H,d)>0$ such that%
\begin{equation}
E[\left\vert \Lambda _{\alpha }^{f}(\theta ,t,z)\right\vert ^{2}]\leq
C^{m+\left\vert \alpha\right\vert}\Psi _{k}^{f}(\theta ,t,z).  \label{supestL}
\end{equation}%
Moreover, we have%
\begin{equation}
\left\vert E[\int_{(\mathbb{R}^{d})^{m}}\Lambda _{\alpha }^{f}(\theta
,t,z)dz]\right\vert \leq C^{m/2+\left\vert \alpha\right\vert/2}\prod_{j=1}^{m}\left\Vert
f_{j}\right\Vert _{L^{1}(\mathbb{R}^{d};L^{\infty }([0,T]))}(\Psi
_{k}^{\varkappa }(\theta ,t))^{1/2}.  \label{intestL}
\end{equation}

\end{thm}

\begin{proof}
For notational convenience we consider $\theta =0$ and set $\Lambda _{\alpha
}^{f}(t,z)=\Lambda _{\alpha }^{f}(0,t,z).$

For an integrable function $g:(\mathbb{R}^{d})^{m}\longrightarrow \mathbb{C}$
we can write%
\begin{eqnarray*}
&&\left\vert \int_{(\mathbb{R}^{d})^{m}}g(u_{1},...,u_{m})du_{1}...du_{m}%
\right\vert ^{2} \\
&=&\int_{(\mathbb{R}^{d})^{m}}g(u_{1},...,u_{m})du_{1}...du_{m}\int_{(%
\mathbb{R}^{d})^{m}}\overline{g(u_{m+1},...,u_{2m})}du_{m+1}...du_{2m} \\
&=&\int_{(\mathbb{R}^{d})^{m}}g(u_{1},...,u_{m})du_{1}...du_{m}(-1)^{dm}%
\int_{(\mathbb{R}^{d})^{m}}\overline{g(-u_{m+1},...,-u_{2m})}%
du_{m+1}...du_{2m},
\end{eqnarray*}%
where we used the change of variables $(u_{m+1},...,u_{2m})\longmapsto
(-u_{m+1},...,-u_{2m})$ in the third equality.

This gives%
\begin{eqnarray*}
&&\left\vert \Lambda _{\alpha }^{f}(\theta ,t,z)\right\vert ^{2} \\
&=&(2\pi )^{-2dm}(-1)^{dm}\int_{(\mathbb{R}^{d})^{2m}}\int_{\Delta
_{0,t}^{m}}\prod_{j=1}^{m}f_{j}(s_{j},z_{j})(-iu_{j})^{\alpha
_{j}}e^{-i\left\langle u_{j},B_{s_{j}}-z_{j}\right\rangle }ds_{1}...ds_{m} \\
&&\times \int_{\Delta
_{0,t}^{m}}\prod_{j=m+1}^{2m}f_{[j]}(s_{j},z_{[j]})(-iu_{j})^{\alpha
_{\lbrack j]}}e^{-i\left\langle u_{j},B_{s_{j}}-z_{[j]}\right\rangle
}ds_{m+1}...ds_{2m}du_{1}...du_{2m} \\
&=&(2\pi )^{-2dm}(-1)^{dm}\sum_{\sigma \in S(m,m)}\int_{(\mathbb{R}%
^{d})^{2m}}\left( \prod_{j=1}^{m}e^{-i\left\langle
z_{j},u_{j}+u_{j+m}\right\rangle }\right) \\
&&\times \int_{\Delta _{0,t}^{2m}}f_{\sigma
}(s,z)\prod_{j=1}^{2m}u_{\sigma (j)}^{\alpha _{\lbrack \sigma
(j)]}}\exp \left\{ -\sum_{j=1}^{2m}\left\langle u_{\sigma
(j)},B_{s_{j}}\right\rangle \right\} ds_{1}...ds_{2m}du_{1}...du_{2m},
\end{eqnarray*}%
where we used (\ref{shuffleIntegral}) in the last step.

Taking the expectation on both sides yields%
\begin{eqnarray}
&&E[\left\vert \Lambda _{\alpha }^{f}(\theta ,t,z)\right\vert ^{2}]
\label{Lambda} \\
&=&(2\pi )^{-2dm}(-1)^{dm}\sum_{\sigma \in S(m,m)}\int_{(\mathbb{R}%
^{d})^{2m}}\left( \prod_{j=1}^{m}e^{-i\left\langle
z_{j},u_{j}+u_{j+m}\right\rangle }\right)   \notag \\
&&\times \int_{\Delta _{0,t}^{2m}}f_{\sigma
}(s,z)\prod_{j=1}^{2m}u_{\sigma (j)}^{\alpha _{\lbrack \sigma
(j)]}}\exp \left\{ -\frac{1}{2}Var[\sum_{j=1}^{2m}\left\langle u_{\sigma
(j)},B_{s_{j}}\right\rangle ]\right\} ds_{1}...ds_{2m}du_{1}...du_{2m} 
\notag \\
&=&(2\pi )^{-2dm}(-1)^{dm}\sum_{\sigma \in S(m,m)}\int_{(\mathbb{R}%
^{d})^{2m}}\left( \prod_{j=1}^{m}e^{-i\left\langle
z_{j},u_{j}+u_{j+m}\right\rangle }\right)   \notag \\
&&\times \int_{\Delta _{0,t}^{2m}}f_{\sigma
}(s,z)\prod_{j=1}^{2m}u_{\sigma (j)}^{\alpha _{\lbrack \sigma
(j)]}}\exp \left\{ -\frac{1}{2}\sum_{l=1}^{d}Var[\sum_{j=1}^{2m}u_{\sigma
(j)}^{(l)}B_{s_{j}}^{(1)}]\right\}
ds_{1}...ds_{2m}du_{1}^{(1)}...du_{2m}^{(1)}  \notag \\
&&...du_{1}^{(d)}...du_{2m}^{(d)}  \notag \\
&=&(2\pi )^{-2dm}(-1)^{dm}\sum_{\sigma \in S(m,m)}\int_{(\mathbb{R}%
^{d})^{2m}}\left( \prod_{j=1}^{m}e^{-i\left\langle
z_{j},u_{j}+u_{j+m}\right\rangle }\right)   \notag \\
&&\times \int_{\Delta _{0,t}^{2m}}f_{\sigma
}(s,z)\prod_{j=1}^{2m}u_{\sigma (j)}^{\alpha _{\lbrack \sigma
(j)]}}\prod_{l=1}^{d}\exp \left\{ -\frac{1}{2}((u_{\sigma
(j)}^{(l)})_{1\leq j\leq 2m})^{T}Q((u_{\sigma (j)}^{(l)})_{1\leq j\leq
2m})\right\} ds_{1}...ds_{2m}  \notag \\
&&du_{\sigma (1)}^{(1)}...du_{\sigma (2m)}^{(1)}...du_{\sigma
(1)}^{(d)}...du_{\sigma (2m)}^{(d)},  \notag
\end{eqnarray}%
where%
\begin{equation*}
Q=Q(s):=(E[B_{s_{i}}^{(1)}B_{s_{j}}^{(1)}])_{1\leq i,j\leq 2m}.
\end{equation*}%
Further, we see that%
\begin{eqnarray}
&&\int_{\Delta _{0,t}^{2m}}\left\vert f_{\sigma }(s,z)\right\vert \int_{(%
\mathbb{R}^{d})^{2m}}\prod_{j=1}^{2m}\prod_{l=1}^{d}\left%
\vert u_{\sigma (j)}^{(l)}\right\vert ^{\alpha _{\lbrack \sigma
(j)]}^{(l)}}\prod_{l=1}^{d}\exp \left\{ -\frac{1}{2}((u_{\sigma
(j)}^{(l)})_{1\leq j\leq 2m})^{T}Q((u_{\sigma (j)}^{(l)})_{1\leq j\leq
2m})\right\}   \notag \\
&&du_{\sigma (1)}^{(1)}...du_{\sigma (2m)}^{(1)}...du_{\sigma
(1)}^{(d)}...du_{\sigma (2m)}^{(d)}ds_{1}...ds_{2m}  \notag \\
&=&\int_{\Delta _{0,t}^{2m}}\left\vert f_{\sigma }(s,z)\right\vert \int_{(%
\mathbb{R}^{d})^{2m}}\prod_{j=1}^{2m}\prod_{l=1}^{d}\left%
\vert u_{j}^{(l)}\right\vert ^{\alpha _{\lbrack \sigma (j)]}^{(l)}}  \notag
\\
&&\times \prod_{l=1}^{d}\exp \left\{ -\frac{1}{2}\left\langle
Qu^{(l)},u^{(l)}\right\rangle \right\}   \notag \\
&&du_{1}^{(1)}...du_{2m}^{(1)}...du_{1}^{(d)}...du_{2m}^{(d)}ds_{1}...ds_{2m}
\notag \\
&=&\int_{\Delta _{0,t}^{2m}}\left\vert f_{\sigma }(s,z)\right\vert
\prod_{l=1}^{d}\int_{\mathbb{R}^{2m}}(\prod_{j=1}^{2m}\left%
\vert u_{j}^{(l)}\right\vert ^{\alpha _{\lbrack \sigma (j)]}^{(l)}})\exp
\left\{ -\frac{1}{2}\left\langle Qu^{(l)},u^{(l)}\right\rangle \right\}
du_{1}^{(l)}...du_{2m}^{(l)}ds_{1}...ds_{2m},  \label{Lambda2}
\end{eqnarray}%
where%
\begin{equation*}
u^{(l)}:=(u_{j}^{(l)})_{1\leq j\leq 2m}.
\end{equation*}%
We have that%
\begin{eqnarray*}
&&\int_{\mathbb{R}^{2m}}(\prod_{j=1}^{2m}\left\vert
u_{j}^{(l)}\right\vert ^{\alpha _{\lbrack \sigma (j)]}^{(l)}})\exp \left\{ -%
\frac{1}{2}\left\langle Qu^{(l)},u^{(l)}\right\rangle \right\}
du_{1}^{(l)}...du_{2m}^{(l)} \\
&=&\frac{1}{(\det Q)^{1/2}}\int_{\mathbb{R}^{2m}}(\prod_{j=1}^{2m}%
\left\vert \left\langle Q^{-1/2}u^{(l)},e_{j}\right\rangle \right\vert
^{\alpha _{\lbrack \sigma (j)]}^{(l)}})\exp \left\{ -\frac{1}{2}\left\langle
u^{(l)},u^{(l)}\right\rangle \right\} du_{1}^{(l)}...du_{2m}^{(l)},
\end{eqnarray*}%
where $e_{i},i=1,...,2m$ is the standard ONB of $\mathbb{R}^{2m}$. 

We also get that%
\begin{eqnarray*}
&&\int_{\mathbb{R}^{2m}}(\prod_{j=1}^{2m}\left\vert \left\langle
Q^{-1/2}u^{(l)},e_{j}\right\rangle \right\vert ^{\alpha _{\lbrack \sigma
(j)]}^{(l)}})\exp \left\{ -\frac{1}{2}\left\langle
u^{(l)},u^{(l)}\right\rangle \right\} du_{1}^{(l)}...du_{2m}^{(l)} \\
&=&(2\pi )^{m}E[\prod_{j=1}^{2m}\left\vert \left\langle
Q^{-1/2}Z,e_{j}\right\rangle \right\vert ^{\alpha _{\lbrack \sigma
(j)]}^{(l)}}],
\end{eqnarray*}%
where%
\begin{equation*}
Z\sim \mathcal{N}(\mathcal{O},I_{2m\times 2m}).
\end{equation*}%
We know from Lemma \ref{LiWei}, which is a type of Brascamp-Lieb
inequality that%
\begin{eqnarray*}
&&E[\prod_{j=1}^{2m}\left\vert \left\langle
Q^{-1/2}Z,e_{j}\right\rangle \right\vert ^{\alpha _{\lbrack \sigma
(j)]}^{(l)}}] \\
&\leq &\sqrt{perm(\sum )}=\sqrt{\sum_{\pi \in S_{2\left\vert \alpha
^{(l)}\right\vert }}\prod_{i=1}^{2\left\vert \alpha
^{(l)}\right\vert }a_{i\pi (i)}},
\end{eqnarray*}%
where $perm(\sum )$ is the permanent of the covariance matrix $\sum =(a_{ij})
$ of the Gaussian random vector%
\begin{equation*}
\underset{\alpha _{\lbrack \sigma (1)]}^{(1)}\text{ times}}{\underbrace{%
(\left\langle Q^{-1/2}Z,e_{1}\right\rangle ,...,\left\langle
Q^{-1/2}Z,e_{1}\right\rangle }},\underset{\alpha _{\lbrack \sigma (2)]}^{(1)}%
\text{ times}}{\underbrace{\left\langle Q^{-1/2}Z,e_{2}\right\rangle
,...,\left\langle Q^{-1/2}Z,e_{2}\right\rangle }},...,\underset{\alpha
_{\lbrack \sigma (2m)]}^{(1)}\text{ times}}{\underbrace{\left\langle
Q^{-1/2}Z,e_{2m}\right\rangle ,...,\left\langle
Q^{-1/2}Z,e_{2m}\right\rangle }}),
\end{equation*}%
$\left\vert \alpha ^{(l)}\right\vert :=\sum_{j=1}^{m}\alpha _{j}^{(l)}$ and
where $S_{n}$ stands for the permutation group of size $n$.

In addition, using an upper bound for the permanent of positive semidefinite
matrices (see \cite{AG}) or
direct computations we get that%
\begin{equation}
perm(\sum )=\sum_{\pi \in S_{2\left\vert \alpha ^{(l)}\right\vert
}}\prod_{i=1}^{2\left\vert \alpha ^{(l)}\right\vert }a_{i\pi
(i)}\leq (2\left\vert \alpha ^{(l)}\right\vert
)!\prod_{i=1}^{2\left\vert \alpha ^{(l)}\right\vert }a_{ii}.
\label{PSD}
\end{equation}

Let now $i\in \lbrack \sum_{r=1}^{j-1}\alpha _{\lbrack \sigma
(r)]}^{(1)}+1,\alpha _{\lbrack \sigma (j)]}^{(1)}]$ for some arbitrary fixed 
$j\in \{1,...,2m\}$. Then%
\begin{equation*}
a_{ii}=E[\left\langle Q^{-1/2}Z,e_{j}\right\rangle \left\langle
Q^{-1/2}Z,e_{j}\right\rangle ].
\end{equation*}

\bigskip Further using substitution, we also have that%
\begin{eqnarray*}
&&E[\left\langle Q^{-1/2}Z,e_{j}\right\rangle \left\langle
Q^{-1/2}Z,e_{j}\right\rangle ] \\
&=&(\det Q)^{1/2}\frac{1}{(2\pi )^{m}}\int_{\mathbb{R}^{2m}}\left\langle
u,e_{j}\right\rangle ^{2}\exp (-\frac{1}{2}\left\langle Qu,u\right\rangle
)du_{1}...du_{2m} \\
&=&(\det Q)^{1/2}\frac{1}{(2\pi )^{m}}\int_{\mathbb{R}^{2m}}u_{j}^{2}\exp (-%
\frac{1}{2}\left\langle Qu,u\right\rangle )du_{1}...du_{2m}
\end{eqnarray*}

\bigskip

We now want to use Lemma \ref{CD}.

\bigskip Then we get that%
\begin{eqnarray*}
&&\int_{\mathbb{R}^{2m}}u_{j}^{2}\exp (-\frac{1}{2}\left\langle
Qu,u\right\rangle )du_{1}...du_{m} \\
&=&\frac{(2\pi )^{(2m-1)/2}}{(\det Q)^{1/2}}\int_{\mathbb{R}}v^{2}\exp (-%
\frac{1}{2}v^{2})dv\frac{1}{\sigma _{j}^{2}} \\
&=&\frac{(2\pi )^{m}}{(\det Q)^{1/2}}\frac{1}{\sigma _{j}^{2}},
\end{eqnarray*}%
where $\sigma _{j}^{2}:=Var[B_{s_{j}}^{H}\left\vert
B_{s_{1}}^{H},...,B_{s_{2m}}^{H}\text{ without }B_{s_{j}}^{H}\right] .$

We now want to use strong local non-determinism of the form (see (\ref{eq:SLND})): For all $t\in
\lbrack 0,T],$ $0<r<t:$%
\begin{equation*}
Var[B_{t}^{H}\left\vert B_{s}^{H},\left\vert t-s\right\vert \geq r\right]
\geq Kr^{2H}.
\end{equation*}%
The latter implies that 
\begin{equation*}
(\det Q(s))^{1/2}\geq K^{(2m-1)/2}\left\vert s_{1}\right\vert ^{H}\left\vert
s_{2}-s_{1}\right\vert ^{H}...\left\vert s_{2m}-s_{2m-1}\right\vert ^{H}
\end{equation*}%
as well as%
\begin{equation*}
\sigma _{j}^{2}\geq K\min \{\left\vert s_{j}-s_{j-1}\right\vert
^{2H},\left\vert s_{j+1}-s_{j}\right\vert ^{2H}\}.
\end{equation*}%
Thus%
\begin{eqnarray*}
\prod_{j=1}^{2m}\sigma _{l}^{-2\alpha _{\lbrack \sigma (j)]}^{(1)}}
&\leq &K^{-2m}\prod_{j=1}^{2m}\frac{1}{\min \{\left\vert
s_{j}-s_{j-1}\right\vert ^{2H\alpha _{\lbrack \sigma (j)]}^{(1)}},\left\vert
s_{j+1}-s_{j}\right\vert ^{2H\alpha _{\lbrack \sigma (j)]}^{(1)}}\}} \\
&\leq &C^{m+\left\vert \alpha ^{(l)}\right\vert}\prod_{j=1}^{2m}\frac{1}{\left\vert
s_{j}-s_{j-1}\right\vert ^{4H\alpha _{\lbrack \sigma (j)]}^{(1)}}}
\end{eqnarray*}%
for a constant $C$ only depending on $H$ and $T$.

\bigskip Hence, it follows from (\ref{PSD}) that%
\begin{eqnarray*}
perm(\sum ) &\leq &(2\left\vert \alpha ^{(l)}\right\vert
)!\prod_{i=1}^{2\left\vert \alpha ^{(l)}\right\vert }a_{ii} \\
&\leq &(2\left\vert \alpha ^{(l)}\right\vert
)!\prod_{j=1}^{2m}((\det Q)^{1/2}\frac{1}{(2\pi )^{m}}\frac{(2\pi
)^{m}}{(\det Q)^{1/2}}\frac{1}{\sigma _{j}^{2}})^{\alpha _{\lbrack \sigma
(j)]}^{(1)}} \\
&\leq &(2\left\vert \alpha ^{(l)}\right\vert )!C^{m+\left\vert \alpha ^{(l)}\right\vert}\prod_{j=1}^{2m}%
\frac{1}{\left\vert s_{j}-s_{j-1}\right\vert ^{4H\alpha _{\lbrack \sigma
(j)]}^{(1)}}}.
\end{eqnarray*}%
So%
\begin{eqnarray*}
&&E[\prod_{j=1}^{2m}\left\vert \left\langle
Q^{-1/2}Z,e_{j}\right\rangle \right\vert ^{\alpha _{\lbrack \sigma
(j)]}^{(l)}}]\leq \sqrt{perm(\sum )} \\
&\leq &\sqrt{(2\left\vert \alpha ^{(l)}\right\vert )!}C^{m+\left\vert \alpha ^{(l)}\right\vert}\prod_{j=1}^{2m}\frac{1}{\left\vert s_{j}-s_{j-1}\right\vert ^{2H\alpha
_{\lbrack \sigma (j)]}^{(1)}}}.
\end{eqnarray*}%
Therefore we obtain from (\ref{Lambda}) and (\ref{Lambda2}) that%
\begin{eqnarray*}
&&E[\left\vert \Lambda _{\alpha }^{f}(\theta ,t,z)\right\vert ^{2}] \\
&\leq &C^{m}\int_{\Delta _{0,t}^{2m}}\left\vert f_{\sigma }(s,z)\right\vert
\prod_{l=1}^{d}\int_{\mathbb{R}^{2m}}(\prod_{j=1}^{2m}\left%
\vert u_{j}^{(l)}\right\vert ^{\alpha _{\lbrack \sigma (j)]}^{(l)}})\exp
\left\{ -\frac{1}{2}\left\langle Qu^{(l)},u^{(l)}\right\rangle \right\}
du_{1}^{(l)}...du_{2m}^{(l)}ds_{1}...ds_{2m} \\
&\leq &M^{m}\int_{\Delta _{0,t}^{2m}}\left\vert f_{\sigma }(s,z)\right\vert 
\frac{1}{(\det Q(s))^{d/2}}\prod_{l=1}^{d}\sqrt{(2\left\vert \alpha
^{(l)}\right\vert )!}C^{m+\left\vert \alpha ^{(l)}\right\vert}\prod_{j=1}^{2m}\frac{1}{\left\vert
s_{j}-s_{j-1}\right\vert ^{2H\alpha _{\lbrack \sigma (j)]}^{(1)}}}%
ds_{1}...ds_{2m} \\
&=&M^{m}C^{md+\left\vert \alpha\right\vert}\prod_{l=1}^{d}\sqrt{(2\left\vert \alpha
^{(l)}\right\vert )!}\int_{\Delta _{0,t}^{2m}}\left\vert f_{\sigma
}(s,z)\right\vert \prod_{j=1}^{2m}\frac{1}{\left\vert
s_{j}-s_{j-1}\right\vert ^{H(d+2\sum_{l=1}^{d}\alpha _{\lbrack \sigma
(j)]}^{(1)})}}ds_{1}...ds_{2m}
\end{eqnarray*}%
for a constant $M$ depending on $d$.

\bigskip Finally, we show estimate (\ref{intestL}). Using the inequality (\ref{supestL}), we find that%
\begin{eqnarray*}
&&\left\vert E\left[ \int_{(\mathbb{R}^{d})^{m}}\Lambda _{\alpha
}^{\varkappa f}(\theta ,t,z)dz\right] \right\vert \\
&\leq &\int_{(\mathbb{R}^{d})^{m}}(E[\left\vert \Lambda _{\alpha
}^{\varkappa f}(\theta ,t,z)\right\vert ^{2})^{1/2}dz\leq C^{m/2+\left\vert \alpha\right\vert/2}\int_{(%
\mathbb{R}^{d})^{m}}(\Psi _{k}^{\varkappa f}(\theta ,t,z))^{1/2}dz.
\end{eqnarray*}%
Taking the supremum over $[0,T]$ for each function $f_{j}$, i.e.%
\begin{equation*}
\left\vert f_{[\sigma (j)]}(s_{j},z_{[\sigma (j)]})\right\vert \leq
\sup_{s_{j}\in \lbrack 0,T]}\left\vert f_{[\sigma (j)]}(s_{j},z_{[\sigma
(j)]})\right\vert ,j=1,...,2m
\end{equation*}%
one obtains that%
\begin{eqnarray*}
&&\left\vert E\left[ \int_{(\mathbb{R}^{d})^{m}}\Lambda _{\alpha
}^{\varkappa f}(\theta ,t,z)dz\right] \right\vert \\
&\leq &C^{m+\left\vert \alpha\right\vert}\max_{\sigma \in S(m,m)}\int_{(\mathbb{R}^{d})^{m}}\left(
\prod_{l=1}^{2m}\left\Vert f_{[\sigma (l)]}(\cdot ,z_{[\sigma
(l)]})\right\Vert _{L^{\infty }([0,T])}\right) ^{1/2}dz \\
&&\times (\prod_{l=1}^{d}\sqrt{(2\left\vert \alpha ^{(l)}\right\vert
)!}\sum_{\sigma \in S(m,m)}\int_{\Delta _{0,t}^{2m}}\left\vert \varkappa
_{\sigma }(s)\right\vert \prod_{j=1}^{2m}\frac{1}{\left\vert
s_{j}-s_{j-1}\right\vert ^{H(d+2\sum_{l=1}^{d}\alpha _{\lbrack \sigma
(j)]}^{(1)})}}ds_{1}...ds_{2m})^{1/2} \\
&=&C^{m+\left\vert \alpha\right\vert}\max_{\sigma \in S(m,m)}\int_{(\mathbb{R}^{d})^{m}}\left(
\prod_{l=1}^{2m}\left\Vert f_{[\sigma (l)]}(\cdot ,z_{[\sigma
(l)]})\right\Vert _{L^{\infty }([0,T])}\right) ^{1/2}dz\cdot (\Psi
_{k}^{\varkappa }(\theta ,t))^{1/2} \\
&=&C^{m+\left\vert \alpha\right\vert}\int_{(\mathbb{R}^{d})^{m}}\prod_{j=1}^{m}\left\Vert
f_{j}(\cdot ,z_{j})\right\Vert _{L^{\infty }([0,T])}dz\cdot (\Psi
_{k}^{\varkappa }(\theta ,t))^{1/2} \\
&=&C^{m+\left\vert \alpha\right\vert}\prod_{j=1}^{m}\left\Vert f_{j}(\cdot ,z_{j})\right\Vert
_{L^{1}(\mathbb{R}^{d};L^{\infty }([0,T]))}\cdot (\Psi _{k}^{\varkappa
}(\theta ,t))^{1/2}.
\end{eqnarray*}

\end{proof}

The next result is a key estimate which shows why fractional Brownian motion regularises \eqref{SDEfBm}. It rests in fact on the earlier integration by parts formula. This estimate is given in more explicit terms when the function $\varkappa$ is chosen to be
\begin{align*}
\varkappa_j(s) = (K_H(s,\theta)-K_H(s,\theta'))^{\varepsilon_j}, \quad \theta < s < t
\end{align*}
and,
\begin{align*}
\varkappa_j(s) = (K_H(s,\theta))^{\varepsilon_j}, \quad \theta < s < t
\end{align*}
for every $j=1,\dots,m$ with $(\varepsilon_1,\dots, \varepsilon_{m})\in \{0,1\}^{m}$. It will be made clear why these choices are important in the forthcoming section.

\begin{prop}\label{mainestimate1}
Let $B^{H},H\in (0,1/2)$ be a standard $d-$dimensional fractional Brownian
motion and functions $f$ and $\varkappa $ as in (\ref{f}), respectively
as in (\ref{kappa}). Let $\theta ,\theta \prime ,t\in \lbrack 0,T],\theta
\prime <\theta <t$ and%
\begin{equation*}
\varkappa _{j}(s)=(K_{H}(s,\theta )-K_{H}(s,\theta \prime ))^{\varepsilon
_{j}},\theta <s<t
\end{equation*}
for every $j=1,...,m$ with $(\varepsilon _{1},...,\varepsilon _{m})\in
\{0,1\}^{m}$ for $\theta ,\theta \prime \in \lbrack 0,T]$ with $\theta
\prime <\theta .$ Let $\alpha \in (\mathbb{N}_{0}^{d})^{m}$ be a
multi-index. If 
\begin{align*}\label{3.7}
H<\frac{\frac{1}{2}-\gamma }{(d+2\sum_{l=1}^{d}\alpha _{
j}^{(1)})}
\end{align*}
for all $j$, where $\gamma \in (0,H)$ is sufficiently small, then there exists a universal constant $C$ (depending on $H$, $T$
and $d$, but independent of $m$, $\{f_{i}\}_{i=1,...,m}$ and $\alpha $) such
that for any $\theta ,t\in \lbrack 0,T]$ with $\theta <t$ we have%
\begin{eqnarray*}
&&\left\vert E\int_{\Delta _{\theta ,t}^{m}}\left(
\prod_{j=1}^{m}D^{\alpha _{j}}f_{j}(s_{j},B_{s_{j}}^{H})\varkappa
_{j}(s_{j})\right) ds\right\vert  \\
&\leq &C^{m+\left\vert \alpha\right\vert}\prod_{j=1}^{m}\left\Vert f_{j}(\cdot ,z_{j})\right\Vert
_{L^{1}(\mathbb{R}^{d};L^{\infty }([0,T]))}\left( \frac{\theta -\theta
\prime }{\theta \theta \prime }\right) ^{\gamma \sum_{j=1}^{m}\varepsilon
_{j}}\theta ^{(H-\frac{1}{2}-\gamma )\sum_{j=1}^{m}\varepsilon _{j}} \\
&&\times \frac{(\prod_{l=1}^{d}(2\left\vert \alpha ^{(l)}\right\vert
)!)^{1/4}(t-\theta )^{-H(md+2\left\vert \alpha \right\vert )-(H-\frac{1}{2}%
-\gamma )\sum_{j=1}^{m}\varepsilon _{j}+m}}{\Gamma (-H(2md+4\left\vert
\alpha \right\vert )+2(H-\frac{1}{2}-\gamma )\sum_{j=1}^{m}\varepsilon
_{j}+2m)^{1/2}}.
\end{eqnarray*}%

\end{prop}

\begin{proof}
By definition of $\Lambda _{\alpha }^{\varkappa f}$ (\ref{LambdaDef}) it
immediately follows that the integral in our proposition can be expressed as%
\begin{equation*}
\int_{\Delta _{\theta ,t}^{m}}\left( \prod_{j=1}^{m}D^{\alpha
_{j}}f_{j}(s_{j},B_{s_{j}}^{H})\varkappa _{j}(s_{j})\right) ds=\int_{\mathbb{%
R}^{dm}}\Lambda _{\alpha }^{\varkappa f}(\theta ,t,z)dz.
\end{equation*}%
Taking expectation and using Theorem \ref{mainthmlocaltime} we obtain%
\begin{equation*}
\left\vert E\int_{\Delta _{\theta ,t}^{m}}\left(
\prod_{j=1}^{m}D^{\alpha _{j}}f_{j}(s_{j},B_{s_{j}}^{H})\varkappa
_{j}(s_{j})\right) ds\right\vert \leq C^{m+\left\vert \alpha \right\vert}\prod_{j=1}^{m}\left\Vert
f_{j}(\cdot ,z_{j})\right\Vert _{L^{1}(\mathbb{R}^{d};L^{\infty
}([0,T]))}\cdot (\Psi _{k}^{\varkappa }(\theta ,t))^{1/2},
\end{equation*}%
where in this situation 
\begin{eqnarray*}
&&\Psi _{k}^{\varkappa }(\theta ,t) \\
&:&=\prod_{l=1}^{d}\sqrt{(2\left\vert \alpha ^{(l)}\right\vert )!}%
\sum_{\sigma \in S(m,m)}\int_{\Delta
_{0,t}^{2m}}\prod_{j=1}^{2m}(K_{H}(s_{j},\theta )-K_{H}(s_{j},\theta
\prime ))^{\varepsilon _{\lbrack \sigma (j)]}} \\
&&\frac{1}{\left\vert s_{j}-s_{j-1}\right\vert ^{H(d+2\sum_{l=1}^{d}\alpha
_{\lbrack \sigma (j)]}^{(1)})}}ds_{1}...ds_{2m}.
\end{eqnarray*}%
We want to apply Lemma \ref{VI_iterativeInt}. For this, we need that $%
-H(d+2\sum_{l=1}^{d}\alpha _{\lbrack \sigma (j)]}^{(1)})+(H-\frac{1}{2}%
-\gamma )\varepsilon _{\lbrack \sigma (j)]}>-1$ for all $j=1,...,2m.$ The
worst case is, when $\varepsilon _{\lbrack \sigma (j)]}=1$ for all $j$. So $%
H<\frac{\frac{1}{2}-\gamma }{(d-1+2\sum_{l=1}^{d}\alpha _{\lbrack \sigma
(j)]}^{(1)})}$ for all $j$.
Hence, we have%
\begin{eqnarray*}
\Psi _{k}^{\varkappa }(\theta ,t) &\leq &\sum_{\sigma \in S(m,m)}\left( 
\frac{\theta -\theta \prime }{\theta \theta \prime }\right) ^{\gamma
\sum_{j=1}^{2m}\varepsilon _{\lbrack \sigma (j)]}}\theta ^{(H-\frac{1}{2}%
-\gamma )\sum_{j=1}^{2m}\varepsilon _{\lbrack \sigma (j)]}} \\
&&\times \prod_{l=1}^{d}\sqrt{(2\left\vert \alpha ^{(l)}\right\vert
)!}\Pi _{\gamma }(2m)(t-\theta )^{-H(2md+4\left\vert \alpha \right\vert )+(H-%
\frac{1}{2}-\gamma )\sum_{j=1}^{2m}\varepsilon _{\lbrack \sigma (j)]}+2m},
\end{eqnarray*}%
where $\Pi _{\gamma }(m)$ is defined as in Lemma \ref{VI_iterativeInt}. The latter can be bounded above as follows
\begin{equation*}
\Pi _{\gamma }(2m)\leq \frac{\prod_{j=1}^{2m}\Gamma
(1-H(d+2\sum_{l=1}^{d}\alpha _{\lbrack \sigma (j)]}^{(1)}))}{\Gamma
(-H(2md+4\left\vert \alpha \right\vert )+(H-\frac{1}{2}-\gamma
)\sum_{j=1}^{2m}\varepsilon _{\lbrack \sigma (j)]}+2m)}.
\end{equation*}%
Observe that $\sum_{j=1}^{2m}\varepsilon _{\lbrack \sigma
(j)]}=2\sum_{j=1}^{m}\varepsilon _{j}.$ Therefore, we have that%
\begin{eqnarray*}
&&(\Psi _{k}^{\varkappa }(\theta ,t))^{1/2} \\
&\leq &C^{m}\left( \frac{\theta -\theta \prime }{\theta \theta \prime }%
\right) ^{\gamma \sum_{j=1}^{m}\varepsilon _{j}}\theta ^{(H-\frac{1}{2}%
-\gamma )\sum_{j=1}^{m}\varepsilon _{j}} \\
&&\times \frac{(\prod_{l=1}^{d}(2\left\vert \alpha ^{(l)}\right\vert
)!)^{1/4}(t-\theta )^{-H(md+2\left\vert \alpha \right\vert )-(H-\frac{1}{2}%
-\gamma )\sum_{j=1}^{m}\varepsilon _{j}+m}}{\Gamma (-H(2md+4\left\vert
\alpha \right\vert )+2(H-\frac{1}{2}-\gamma )\sum_{j=1}^{m}\varepsilon
_{j}+2m)^{1/2}},
\end{eqnarray*}%
where we used $\prod_{j=1}^{2m}\Gamma (1-H(d+2\sum_{l=1}^{d}\alpha _{\lbrack \sigma (j)]}^{(1)})\leq C^{m}$ for a large
enough constant $C>0$ and $\sqrt{a_{1}+...+a_{m}}\leq \sqrt{a_{1}}+...\sqrt{%
a_{m}}$ for arbitrary non-negative numbers $a_{1},...,a_{m}$.

\end{proof}

\begin{prop}\label{mainestimate2}
Let $B^{H},H\in (0,1/2)$ be a standard $d-$dimensional fractional Brownian
motion and functions $f$ and $\varkappa $ as in (\ref{f}), respectively
as in (\ref{kappa}). Let $\theta ,t\in \lbrack 0,T]$ with $\theta <t$ and%
\begin{equation*}
\varkappa _{j}(s)=(K_{H}(s,\theta ))^{\varepsilon _{j}},\theta <s<t
\end{equation*}%
for every $j=1,...,m$ with $(\varepsilon _{1},...,\varepsilon _{m})\in
\{0,1\}^{m}$ for $\theta ,\theta \prime \in \lbrack 0,T]$ with $\theta
\prime <\theta .$ Let $\alpha \in (\mathbb{N}_{0}^{d})^{m}$ be a
multi-index. If 
\begin{align*}
H<\frac{\frac{1}{2}-\gamma }{(d+2\sum_{l=1}^{d}\alpha _{j}^{(1)})}
\end{align*}
for all $j$, where $\gamma \in (0,H)$ is sufficiently small, then there exists a universal constant $C$ (depending on $H$, $T$
and $d$, but independent of $m$, $\{f_{i}\}_{i=1,...,m}$ and $\alpha $) such
that for any $\theta ,t\in \lbrack 0,T]$ with $\theta <t$ we have%
\begin{eqnarray*}
&&\left\vert E\int_{\Delta _{\theta ,t}^{m}}\left(
\prod_{j=1}^{m}D^{\alpha _{j}}f_{j}(s_{j},B_{s_{j}}^{H})\varkappa
_{j}(s_{j})\right) ds\right\vert  \\
&\leq &C^{m+\left\vert \alpha\right\vert}\prod_{j=1}^{m}\left\Vert f_{j}(\cdot ,z_{j})\right\Vert
_{L^{1}(\mathbb{R}^{d};L^{\infty }([0,T]))}\left( \frac{\theta -\theta
\prime }{\theta \theta \prime }\right) ^{\gamma \sum_{j=1}^{m}\varepsilon
_{j}}\theta ^{(H-\frac{1}{2}-\gamma )\sum_{j=1}^{m}\varepsilon _{j}} \\
&&\times \frac{(\prod_{l=1}^{d}(2\left\vert \alpha ^{(l)}\right\vert
)!)^{1/4}(t-\theta )^{-H(md+2\left\vert \alpha \right\vert )-(H-\frac{1}{2}%
-\gamma )\sum_{j=1}^{m}\varepsilon _{j}+m}}{\Gamma (-H(2md+4\left\vert
\alpha \right\vert )+2(H-\frac{1}{2}-\gamma )\sum_{j=1}^{m}\varepsilon
_{j}+2m)^{1/2}}.
\end{eqnarray*}%

\end{prop}
\begin{proof}
The proof is similar to the previous proposition.
\end{proof}

\begin{rem}\label{Remark 3.4}
We mention that%
\begin{equation*}
\prod_{l=1}^{d}(2\left\vert \alpha ^{(l)}\right\vert )!\leq
(2\left\vert \alpha \right\vert )!C^{\left\vert \alpha \right\vert }
\end{equation*}%
for a constant $C$ depending on $d$. Later on in the paper, when we deal
with the existence of strong solutions, we will consider the case%
\begin{equation*}
\alpha _{j}^{(l)}\in \{0,1\}\text{ for all }j,l
\end{equation*}%
with%
\begin{equation*}
\left\vert \alpha \right\vert =m.
\end{equation*}

\end{rem}


\section{Local times of a fractional Brownian motion and properties}

One can define, heuristically, the local time $L_{t}^{x}\left(B^{H}\right)$
of $B^{H}$ at $x\in\mathbb{R}^{d}$ by
\[
L_{t}^{x}\left(B^{H}\right)=\int_{0}^{t}\delta_{x}(B_{s}^{H})ds.
\]
 It is known that $L_{t}^{x}\left(B^{H}\right)$ exists and is jointly
continuous in $\left(t,x\right)$ as long as $Hd<1$. See e.g. \cite{pitt.78} and the references therein. Moreover, by
the self-similarity property of the fBm one has that $L_{t}^{x}\left(B^{H}\right)\overset{law}{=}t^{1-Hd}L_{1}^{x/t^{H}}(B^{H})$
and, in particular 
\[
L_{t}^{0}\left(B^{H}\right)\overset{law}{=}t^{1-Hd}L_{1}^{0}(B^{H}).
\]
The rigorous construction of $L_{t}^{x}\left(B^{H}\right)$ involves
approximating the Dirac delta function by an approximate unity. It
is convenient to consider the Gaussian approximation of unity 
\[
\varphi_{\varepsilon}(x)=\varepsilon^{-\frac{d}{2}}\varphi\left(\varepsilon^{-\frac{1}{2}}x\right),\quad \varepsilon>0,
\]
for every $x\in \R^d$ where $\varphi$ is the $d$-dimensional standard Gaussian density.
Then, we can define the smoothed local times 
\[
L_{t}^{x}\left(B^{H},\varepsilon\right)=\int_{0}^{t}\varphi_{\varepsilon}(B_{s}^{H}-x)ds
\]
and construct $L_{t}^{x}\left(B^{H}\right)$ as the limit when $\varepsilon$
tends to zero in $L^{2}\left(\Omega\right)$. Note that, using the
Fourier transform, one can write $\varphi_{\varepsilon}(x)$ as follows
\[
\varphi_{\varepsilon}(x)=\frac{1}{\left(2\pi\right)^{d}}\int_{\mathbb{R}^{d}}\exp\left(i\left\langle \xi,x\right\rangle _{\mathbb{R}^{d}}-\varepsilon\frac{\left|\xi\right|_{\mathbb{R}^{d}}^{2}}{2}\right)d\xi.
\]
The previous expression allows us to write 
\[
L_{t}^{x}\left(B^{H},\varepsilon\right)=\frac{1}{\left(2\pi\right)^{d}}\int_{0}^{t}\int_{\mathbb{R}^{d}}\exp\left(i\left\langle \xi,B_{s}^{H}-x\right\rangle _{\mathbb{R}^{d}}-\varepsilon\frac{\left|\xi\right|_{\mathbb{R}^{d}}^{2}}{2}\right)d\xi ds,
\]
and
\begin{eqnarray}
\mathbb{E}\left[L_{t}^{x}\left(B^{H},\varepsilon\right)^{m}\right] & = & \frac{m!}{\left(2\pi\right)^{md}}\int_{\mathcal{T}_{m}(0,t)}\int_{\mathbb{R}^{md}}\mathbb{E}\left[\exp\left(i\sum_{j=1}^{m}\left\langle \xi_{j},B_{s_{j}}^{H}\right\rangle _{\mathbb{R}^{d}}\right)\right]\nonumber \\
 &  & \times\exp\left(-\sum_{j=1}^{m}\left(i\left\langle \xi_{j},x\right\rangle _{\mathbb{R}^{d}}+\frac{\varepsilon\left|\xi_{j}\right|_{\mathbb{R}^{d}}^{2}}{2}\right)\right)d\mathbf{\bar{\xi}}d\mathbf{s},\label{eq:MomentSmoothLT}
\end{eqnarray}
where $\mathbf{\bar{\xi}}=(\xi_{1},...,\xi_{m})=(\xi_{1}^{1},...,\xi_{1}^{d},\ldots,\xi_{m}^{1}...,\xi_{m}^{d})\in\mathbb{R}^{md}$
and $\mathbf{s}=\left(s_{1},...,s_{m}\right)\in\mathcal{T}_{m}(0,t)=\left\{ 0\leq s_{1}<s_{2}<\cdots<s\leq t\right\} $.
Next, note that 
\begin{eqnarray*}
\mathbb{E}\left[\exp\left(i\sum_{j=1}^{m}\left\langle \xi_{j},B_{s_{j}}^{H}\right\rangle _{\mathbb{R}^{d}}\right)\right] & = & \exp\left(-\frac{1}{2}\mathrm{Var}\left[\sum_{j=1}^{m}\sum_{k=1}^{d}\xi_{j}^{k}B_{s_{j}}^{H,k}\right]\right)\\
 & = & \exp\left(-\frac{1}{2}\sum_{k=1}^{d}\mathrm{Var}\left[\sum_{j=1}^{m}\xi_{j}^{k}B_{s_{j}}^{H,k}\right]\right)\\
 & = & \exp\left(-\frac{1}{2}\sum_{k=1}^{d}\mathrm{\left\langle \xi^{k},Q(\mathbf{s})\xi^{k}\right\rangle }_{\mathbb{R}{}^{m}}\right),
\end{eqnarray*}
where $\xi^{k}=\left(\xi_{1}^{k},...,\xi_{m}^{k}\right)$ and $Q(\mathbf{s})$
is the covariance matrix of the vector $\left(B_{s_{1}}^{H,1},...,B_{s_{m}}^{H,1}\right)$.
Rearranging the terms in the second exponential in equation $\left(\ref{eq:MomentSmoothLT}\right)$we
can write 
\begin{eqnarray*}
\mathbb{E}\left[L_{t}^{x}\left(B^{H},\varepsilon\right)^{m}\right] & = & \frac{m!}{\left(2\pi\right)^{md}}\int_{\mathcal{T}_{m}(0,t)}\int_{\mathbb{R}^{md}}\exp\left(-\frac{1}{2}\sum_{k=1}^{d}\left(\mathrm{\left\langle \xi^{k},Q(\mathbf{s})\xi^{k}\right\rangle }_{\mathbb{R}{}^{m}}+\frac{\varepsilon\left|\xi^{k}\right|_{\mathbb{R}{}^{m}}^{2}}{2}\right)\right)\\
 &  & \times\exp\left(-i\sum_{j=1}^{m}\left\langle \xi_{j},x\right\rangle _{d}\right)d\mathbf{\bar{\xi}}d\mathbf{s},\\
 & \leq & \frac{m!}{\left(2\pi\right)^{md}}\int_{\mathcal{T}_{m}(0,t)}\left(\int_{\mathbb{R}^{m}}\exp\left(-\frac{1}{2}\mathrm{\left\langle \xi^{1},Q(\mathbf{s})\xi^{1}\right\rangle }_{\mathbb{R}{}^{m}}-\frac{\varepsilon\left|\xi^{1}\right|_{\mathbb{R}{}^{m}}^{2}}{2}\right)d\xi^{1}\right)^{d}d\mathbf{s}\\
 & \leq & \frac{m!}{\left(2\pi\right)^{md}}\int_{\mathcal{T}_{m}(0,t)}\left(\int_{\mathbb{R}^{m}}\exp\left(-\frac{1}{2}\mathrm{\left\langle \xi^{1},Q(\mathbf{s})\xi^{1}\right\rangle }\right)d\xi^{1}\right)^{d}d\mathbf{s}\\
 & = & \frac{m!}{\left(2\pi\right)^{\frac{dm}{2}}}\int_{\mathcal{T}_{m}(0,t)}\left(\det Q(\mathbf{s})\right)^{-\frac{d}{2}}d\mathbf{s}\triangleq\alpha_{m}.
\end{eqnarray*}
Hence, by dominated convergence, we can conclude that $\mathbb{E}\left[L_{t}^{x}\left(B^{H},\varepsilon\right)^{m}\right]$
converges when $\varepsilon$ tends to zero as long as $\alpha_{m}<\infty.$
If $\alpha_{2}<\infty$, then one can similarly show that
\begin{align*}
\lim_{\varepsilon_{1},\varepsilon_{2}\rightarrow 0+}\mathbb{E}\left[L_{t}^{x}\left(B^{H},\varepsilon_{1}\right)L_{t}^{x}\left(B^{H},\varepsilon_{2}\right)\right]
\end{align*}
exists, which yields the convergence in $L^{2}\left(\Omega\right)$
of $L_{t}^{x}\left(B^{H},\varepsilon\right).$ If $\alpha_{m}<\infty$
for all $m\geq1$ one can deduce the convergence in $L^{p}\left(\Omega\right),p\geq2$
of $L_{t}^{x}\left(B^{H},\varepsilon\right)$. 

The following well known result can be found in Anderson \cite[p. 42]{A58}.
\begin{lem}
\label{lem:detCov}Let $\left(X_{1},\ldots,X_{m}\right)$ be a mean-zero
Gaussian random vector. Then,
\[
\det\left(\mathrm{Cov}\left[X_{1},\ldots,X_{m}\right]\right)=\mathrm{Var}\left[X_{1}\right]\mathrm{Var}\left[X_{2}|X_{1}\right]\cdots\mathrm{Var}\left[X_{m}|X_{m-1},\ldots,X_{1}\right].
\]

\end{lem}
Another useful elementary result is:
\begin{lem}
\label{lem:CondVar}Let $X$ be a square integrable random variable
and $\mathcal{G}_{1}\subset\mathcal{G}_{2}$ be two $\sigma$-algebras.
Then,
\[
\mathrm{Var}\left[X|\mathcal{G}_{1}\right]\geq\mathrm{Var}\left[X|\mathcal{G}_{2}\right].
\]

\end{lem}
Combining Lemmas \ref{lem:detCov}, \ref{lem:CondVar} and (\ref{eq:SLND})
we get that
\begin{eqnarray*}
\det Q(\mathbf{s}) & = & \mathrm{Var}\left[B_{s_{1}}^{H,1}\right]\mathrm{Var}\left[B_{s_{2}}^{H,1}|B_{s_{1}}^{H,1}\right]\cdots\mathrm{Var}\left[B_{s_{m}}^{H,1}|B_{s_{m-1}}^{H,1},\ldots,B_{s_{1}}^{H,1}\right]\\
 & \geq & s_{1}^{2H}\mathrm{Var}\left[B_{s_{2}}^{H,1}|\mathcal{F}_{s_{1}}\right]\cdots\mathrm{Var}\left[B_{s_{m}}^{H,1}|\mathcal{F}_{s_{m-1}}\right]\\
 & \geq & K^{m-1}s_{1}^{2H}\left(s_{2}-s_{1}\right)^{2H}\cdots\left(s_{m}-s_{m-1}\right)^{2H}
\end{eqnarray*}
and, therefore, 
\begin{align*}
\int_{\mathcal{T_{\textnormal{\ensuremath{m}}}}\left(0,t\right)}\left(\det Q(\mathbf{s})\right)^{-\frac{d}{2}}d\mathbf{s} & \leq K^{\frac{d}{2}(1-m)}\int_{\mathcal{T_{\textnormal{\ensuremath{m}}}}\left(0,t\right)}s_{1}^{-Hd}\left(s_{2}-s_{1}\right)^{-Hd}\cdots\left(s_{m}-s_{m-1}\right)^{-Hd}d\mathbf{s}\\
 & =K^{\frac{d}{2}(1-m)}\left(\prod_{j=1}^{m}\mathcal{B}\left(j\left(1-Hd\right),1-Hd\right)\right)t^{m\left(1-Hd\right)}<\infty,
\end{align*}
if $Hd$$<1$. Finally, we have proved the bound 
\begin{align}
\mathbb{E}\left[L_{t}^{x}\left(B^{H}\right)^{m}\right] & \leq\frac{m!}{\left(2\pi\right)^{\frac{dm}{2}}}K^{\frac{d}{2}(1-m)}\left(\prod_{j=1}^{m}\mathcal{B}\left(j\left(1-Hd\right),1-Hd\right)\right)t^{m\left(1-Hd\right)}\label{eq:BoundMomentLT}
\end{align}

\begin{rem}
We just have checked that if $Hd<1$ then $L_{t}^{x}\left(B^{H}\right)$
exists and has moments of all orders. By checking that $\sum_{m\geq1}\frac{\alpha_{m}}{m!}<\infty$,
one can deduce that $L_{t}^{x}\left(B^{H}\right)$ has exponential
moments or all orders. Furthermore, one can also show the existence of exponential
moments of $L_{t}^{x}\left(B^{H}\right)^{2}$ by doing similar computations as before. However, one may also 
use Theorem \ref{thm:LDforfBM} below to show that the exponential
moments are finite.
\end{rem}
Chen et al. \cite{Ch et al 11} proved the following result on large
deviations for local times of fractional Brownian motion, which we won't use in our paper but which is of independent interest:
\begin{thm}
\label{thm:LDforfBM}Let $B^{H}$ be a standard fractional Brownian
motion with Hurst index $H$ such that $Hd<1$. Then the limit
\[
\lim_{a\rightarrow\infty}a^{-\frac{1}{Hd}}\log P\left(L_{1}^{0}(B^{H})\geq a\right)=-\theta(H,d),
\]
exists and \textup{$\theta(H,d)$ satisfies the following bounds
\[
\left(\frac{\pi c_{H}^{2}}{H}\right)^{\frac{1}{2H}}\theta_{0}(Hd)\leq\theta(H,d)\leq\left(2\pi\right)^{\frac{1}{2H}}\theta_{0}(Hd),
\]
where $c_{H}$ is given by and 
\[
\theta_{0}(\lambda)=\lambda\left(\frac{(1-\lambda)^{1-\lambda}}{\Gamma(1-\lambda)}\right)^{1/\lambda}.
\]
}
\end{thm}


\section{Existence of strong solutions}

As outlined in the introduction the object of study is a \emph{generalized} SDE with additive $d$-dimensional fractional Brownian noise $B^H$ with Hurst parameter $H\in (0,1/2)$, i.e.
\begin{equation}
X_{t}^{x}=x+\alpha L_{t}(X^{x})\cdot \boldsymbol{1}_{d}+B_{t}^{H},0\leq
t\leq T,x\in \mathbb{R}^{d},  \label{SDE}
\end{equation}%

where $L_t(X^x)$, $t\in [0,T]$ is a stochastic process of bounded variation which arises from taking the limit
$$L_t(X^x) := \lim_{\varepsilon \searrow 0} \int_0^t \varphi_{\varepsilon} (X_s^x) ds,$$
in probability, where $\varphi_{\varepsilon}$ are probability densities
approximating $\delta_{0}$, denoting $\delta_0$ the Dirac delta generalized function with total mass at 0. We will consider
\begin{equation}
\varphi_{\varepsilon}(x)=\varepsilon^{-\frac{d}{2}}\varphi(\varepsilon^{-\frac{1}{2}}x),\quad\varepsilon>0,\label{eq:GaussianKernels}
\end{equation}
where $\varphi$ is the $d$-dimensional standard Gaussian density function.

Hereunder, we establish the main result of this section.

\begin{thm}\label{mainthm}
If $H< 1/(2(2+d))$, $d\geq 1$ there exists a continuous strong solution $X^x =\{X_t^x , t\in [0,T], x\in \R^d\}$ of equation \eqref{SDE} for all $\alpha$. Moreover, for every $t\in [0,T]$, $X_t$ is Malliavin differentiable in the direction of the Brownian motion $W$ in \eqref{WBH}.
\end{thm}

\begin{prop}\label{mainprop}
Retain the conditions of Theorem \ref{mainthm}. Let $Y_{\cdot }^{x}$ be another solution
to the SDE (\ref{SDE}). Suppose that the Dolean-Dade exponentials
\begin{equation*}
\mathcal{E}(\int_{0}^{T}-K_{H}^{-1}(\int_{0}^{\cdot }\varphi _{\varepsilon
}(Y_{u}^{x})\boldsymbol{1}_{d}du)^{\ast }(s)dW_{s}),\varepsilon >0
\end{equation*}%
converge in $L^{p}(\Omega )$ for $\varepsilon \longrightarrow 0$ for all $%
p\geq 1$, where $\varphi _{\varepsilon }$ is the approximation of the Dirac
delta $\delta _{0}$ in \ref{eq:GaussianKernels} and $\ast $ denotes transposition. Then strong
uniqueness holds for such solutions.\newline In particular, this is the
case, if e.g. uniqueness in law is satisfied.
\end{prop}
The proof of Theorem \eqref{mainthm} essentially consists of four steps:
\begin{enumerate}
\item In the first step, we construct a weak solution $X$ to \eqref{SDE}
by using the version of Girsanov's theorem for the fractional Brownian motion, that is we consider a probability
space $(\Omega,\mathfrak{A},P)$ on which a fractional Brownian
motion $B^{H}$ and a process $X^{x}$ are defined such that \eqref{SDE} holds. However, a priori the solution is not a measurable functional of the the driving noise, that is $X^x$ is not adapted to the filtration $\mathcal{F}=\{\mathcal{F}_{t}\}_{t\in[0,T]}$ generated by $B^{H}$.
\item In the next step, we approximate the generalized drift coefficient $\delta_{0}$ by the Gaussian
kernels $\varphi_{\varepsilon}$. Using classical Picard iteration, we
know that for each smooth coefficient $\varphi_{\varepsilon}$, $\varepsilon>0$,
there exists unique strong solution $X_{\cdot}^{\varepsilon}$ to
the SDE 
\begin{align}
dX_{t}^{\varepsilon}=\alpha\varphi_{\varepsilon}(X_{t}^{\varepsilon})\cdot \boldsymbol{1}_{d}dt+dB_{t}^{H},\,\,0\leq t\leq T,\,\,\,X_{0}^{\varepsilon}=x\in\mathbb{R}^{d}\,.\label{Xn}
\end{align}
Then we prove that for each $t\in[0,T]$ the family $\{X_{t}^{\varepsilon}\}_{\varepsilon>0}$ converges weakly as $\varepsilon \searrow 0$ to the conditional expectation $E[X_{t}|\mathcal{F}_{t}]$ in the space $L^{2}(\Omega;\mathcal{F}_{t})$ of square integrable,
$\mathcal{F}_{t}$-measurable random variables. 
\item Further, it is well known, see e.g. \cite{Nua10}, that for each $t\in[0,T]$
the strong solution $X_{t}^{\varepsilon}$, $\varepsilon>0$, is Malliavin
differentiable, and that the Malliavin derivative $D_{s}X_{t}^{\varepsilon}$,
$0\leq s\leq t$, with respect to $W$ in \eqref{WBH} solves the equation
\begin{align}\label{DXn1}
D_{s}X_{t}^{\varepsilon}=K_{H}(t,s)I_{d}+\int_{s}^{t}\alpha\varphi_{\varepsilon}'(X_{u}^{\varepsilon})\cdot \boldsymbol{1}_{d}D_{s}X_{u}^{\varepsilon}du,
\end{align}
where $\varphi_{\varepsilon}'$ denotes the Jacobian of $\varphi_{\varepsilon}$.
Using a compactness criterion based on Malliavin
calculus (see Appendix A) we then show that for every $t\in[0,T]$ the set of random variables
$\{X_{t}^{\varepsilon}\}_{\varepsilon > 0}$ is relatively compact in $L^{2}(\Omega)$,
which enables us to conclude that $X_{t}^{\varepsilon}$ converges
strongly as $\varepsilon \searrow 0$ in $L^{2}(\Omega;\mathcal{F}_{t})$ to $\mathbb{E}\left[X_{t}|\mathcal{F}_{t}\right]$.
As a consequence of the compactness criterion we also observe that $E[X_{t}|\mathcal{F}_{t}]$ is Malliavin differentiable. 
\item Finally, we prove that $\mathbb{E}\left[X_{t}|\mathcal{F}_{t}\right]=X_{t}$,
which entails that $X_{t}$ is $\mathcal{F}_{t}$-measurable and thus
a strong solution on our specific probability space, on which we assumed our weak solution. 
\end{enumerate}

Let us first have a look at step 1 of our programme, that is we want to construct weak
solutions of \eqref{SDE} by using Girsanov's theorem. Let $(\Omega,\mathfrak{A},\widetilde{P})$ be some
given probability space which carries a $d$-dimensional fractional
Brownian motion $\widetilde{B}^{H}$ with Hurst parameter $H\in(0,1/2)$
and set $X_{t}^{x}:=x+\widetilde{B}_{t}^{H}$, $t\in\left[0,T\right]$,
$x\in\mathbb{R}^{d}$. Set $\theta_{t}:=\left(K_{H}^{-1}\left(\int_{0}^{\cdot}\delta_{0}(X_{r}^{x})dr\boldsymbol{1}_{d}\right)\right)(t)$
and consider the Dol\'{e}ans-Dade exponential 
\begin{align*}
\xi_{t} & := \exp\left\{ \int_{0}^{t}\theta_{s}^{T}dW_{s}-\frac{1}{2}\int_{0}^{t}\theta_{s}^{T}\theta_{s}ds\right\} ,\quad t\in[0,T].
\end{align*}
formally.

If we were allowed to implement Girsanov's theorem in this setting we would arrive at the conclusion that the process
\begin{align}\label{weak}
B_{t}^{H} :=& X_{t}^{x}-x-\int_{0}^{t}\delta_{x}(X_{s}^{x})ds\boldsymbol{1}_{d}\\
= & \widetilde{B}_{t}^{H}-\int_{0}^{t}\delta_{0}(B_{s}^{H})ds\boldsymbol{1}_{d}\nonumber 
\end{align}
is a fractional Brownian motion on $(\Omega,\mathfrak{A},P)$ with
Hurst parameter $H\in(0,1/2)$, where $\frac{dP}{d\widetilde{P}}=\xi_{T}$.
Hence, because of \eqref{weak}, the couple $(X^{x},B^{H})$ will
be a weak solution of \ref{SDE} on $(\Omega,\mathfrak{A},P)$.

Therefore, in what follows we show that the requirements of Theorem \ref{girsanov} are accomplished.

\begin{lem}\label{expmom}
Let $x\in \mathbb{R}^{d}$. If $H<\frac{1}{2(1+d)}$ then%
\begin{equation*}
\sup_{\varepsilon >0}E[\exp (\mu \int_{0}^{T}(K_{H}^{-1}(\int_{0}^{.}\varphi
_{x,\varepsilon }(B_{u}^{H})du)(t))^{2}dt)]<\infty\label{eq:ExponentialMoments-1}
\end{equation*}
for all $\mu \in \mathbb{R},$ where%
\begin{equation*}
\varphi _{x,\varepsilon }(B_{u}^{H})=\frac{1}{(2\pi \varepsilon )^{\frac{d}{2%
}}}\exp (-\frac{\left\vert B_{u}^{H}-x\right\vert _{\mathbb{R}^{d}}^{2}}{%
2\varepsilon }).
\end{equation*}

\end{lem}
\begin{proof}
In order to prove Lemma \ref{expmom}, we can write
\begin{eqnarray*}
K_{H}^{-1}\left(\int_{0}^{\cdot}\varphi _{x,\varepsilon }(B_{r}^{H})dr\right)(t) & = & t^{H-\frac{1}{2}}I_{0+}^{\frac{1}{2}-H}t^{\frac{1}{2}-H}\left(\int_{0}^{\cdot}\varphi _{x,\varepsilon }(B_{r}^{H})dr\right)^{\prime}\left(t\right)\\
 & = & t^{H-\frac{1}{2}}\int_{0}^{t}\gamma_{-\frac{1}{2}-H,\frac{1}{2}-H}(t,u)\varphi _{x,\varepsilon }(B_{u}^{H})du,
\end{eqnarray*}
where 
\[
\gamma_{\alpha,\beta}(t,u)=\left(t-u\right)^{\alpha}u^{\beta}.
\]
Using the self-similarity of the fBm we can write%
\begin{equation*}
K_{H}^{-1}(\int_{0}^{.}\varphi _{x,\varepsilon }(B_{r}^{H})dr)(t)\overset{law%
}{=}t^{\frac{1}{2}-H(1+d)}\int_{0}^{1}\gamma _{-\frac{1}{2}-H,\frac{1}{2}%
-H}(1,u)\varphi _{xt^{-H},\varepsilon (t)}(B_{u}^{H})du,
\end{equation*}%
where $\varepsilon (t):=\varepsilon t^{-2H}$, and hence%
\begin{eqnarray*}
&&K_{H}^{-1}(\int_{0}^{.}\varphi _{x,\varepsilon }(B_{r}^{H})dr)^{2m}(t)%
\overset{law}{=}t^{2m(\frac{1}{2}-H(1+d))}(\int_{0}^{1}\gamma _{-\frac{1}{2}-H,\frac{1}{%
2}-H}(1,u)\varphi _{xt^{-H},\varepsilon (t)}(B_{u}^{H})du)^{2m} \\
&=&t^{2m(\frac{1}{2}-H(1+d))}(2m)!\int_{\mathcal{T}_{2m}(0,1)}\prod_{j=1}^{2m}\gamma
_{-\frac{1}{2}-H,\frac{1}{2}-H}(1,u_{j})\varphi _{xt^{-H},\varepsilon
(t)}(B_{u_{j}}^{H})d\mathbf{u,}
\end{eqnarray*}%
where $\mathcal{T}_{n}(0,s)=\{0\leq u_{1}<u_{2}<...<u_{n}\leq s\}$ and 
\begin{equation*}
\varphi _{xt^{-H},\varepsilon (t)}(B_{u_{j}}^{H})=\frac{1}{(2\pi )^{d}}\int_{%
\mathbb{R}^{d}}\exp (i\left\langle \xi ,B_{u_{j}}^{H}-xt^{-H}\right\rangle _{%
\mathbb{R}^{d}}-\varepsilon (t)\frac{\left\vert \xi \right\vert _{\mathbb{R}%
^{d}}^{2}}{2})d\xi .
\end{equation*}%
Then%
\begin{eqnarray*}
&&\mathbb{E}[(\int_{0}^{T}(K_{H}^{-1}(\int_{0}^{.}\varphi _{x,\varepsilon
}(B_{u}^{H})du)(t))^{2}dt)^{m}] \\
&\leq &T^{m-1}\int_{0}^{T}E[K_{H}^{-1}(\int_{0}^{.}\varphi _{x,\varepsilon
}(B_{r}^{H})dr)^{2m}(t)]dt \\
&=&T^{m-1}\int_{0}^{T}t^{2m(\frac{1}{2}-H(1+d))}(2m)!\int_{\mathcal{T}%
_{2m}(0,1)}(\prod_{j=1}^{2m}\gamma _{-\frac{1}{2}-H,\frac{1}{2}%
-H}(1,u_{j}))E[\prod_{j=1}^{2m}\varphi _{xt^{-H},\varepsilon
(t)}(B_{u_{j}}^{H})]d\mathbf{u}dt.
\end{eqnarray*}%
Moreover,%
\begin{eqnarray*}
&&\mathbb{E}[\prod_{j=1}^{2m}\varphi _{xt^{-H},\varepsilon
(t)}(B_{u_{j}}^{H})] \\
&=&\frac{1}{(2\pi )^{2dm}}\mathbb{E}[\prod_{j=1}^{2m}\int_{\mathbb{R}%
^{d}}\exp (i\left\langle \xi _{j},B_{u_{j}}^{H}-xt^{-H}\right\rangle _{%
\mathbb{R}^{d}}-\varepsilon (t)\frac{\left\vert \xi _{j}\right\vert _{%
\mathbb{R}^{d}}^{2}}{2})d\xi _{j}] \\
&=&\frac{1}{(2\pi )^{2dm}}\int_{\mathbb{R}^{2dm}}E[\exp
(i\sum_{j=1}^{2m}\left\langle \xi _{j},B_{u_{j}}^{H}\right\rangle _{\mathbb{R%
}^{d}})] \\
&&\times \exp (-i\sum_{j=1}^{2m}\left\langle \xi _{j},xt^{-H}\right\rangle _{%
\mathbb{R}^{d}}) \\
&&\times \exp (-\frac{\varepsilon (t)}{2}\sum_{j=1}^{2m}\left\vert \xi
_{j}\right\vert _{\mathbb{R}^{d}}^{2})d\xi _{1}...d\xi _{2m}.
\end{eqnarray*}%
Next note that%
\begin{eqnarray*}
&&\mathbb{E}[\exp (i\sum_{j=1}^{2m}\left\langle \xi
_{j},B_{u_{j}}^{H}\right\rangle _{\mathbb{R}^{d}})] \\
&=&\exp (-\frac{1}{2}Var\left[ \sum_{j=1}^{2m}\sum_{k=1}^{d}\xi
_{j}^{k}B_{u_{j}}^{H,k}\right] ) \\
&=&\exp (-\frac{1}{2}\sum_{k=1}^{d}Var\left[ \sum_{j=1}^{2m}\xi
_{j}^{k}B_{u_{j}}^{H,k}\right] ) \\
&=&\exp (-\frac{1}{2}\sum_{k=1}^{d}\left\langle \xi ^{k},Q(\mathbf{u})\xi
^{k}\right\rangle _{\mathbb{R}^{2m}}),
\end{eqnarray*}%
where%
\begin{equation*}
Q(\mathbf{u})=Cov(B_{u_{1}}^{H,1},...,B_{u_{2m}}^{H,1}).
\end{equation*}%
Hence,%
\begin{eqnarray*}
&&\mathbb{E}[\prod_{j=1}^{2m}\varphi _{xt^{-H},\varepsilon
(t)}(B_{u_{j}}^{H})] \\
&\leq &\frac{1}{(2\pi )^{2dm}}\int_{\mathbb{R}^{2dm}}\exp (-\frac{1}{2}Var%
\left[ \sum_{j=1}^{2m}\sum_{k=1}^{d}\xi _{j}^{k}B_{u_{j}}^{H,k}\right] )d\xi
_{1}...d\xi _{2m} \\
&=&\frac{1}{(2\pi )^{2dm}}\int_{\mathbb{R}^{2dm}}\exp (-\frac{1}{2}%
\sum_{k=1}^{d}\left\langle \xi ^{k},Q(\mathbf{u})\xi ^{k}\right\rangle _{%
\mathbb{R}^{2m}})d\xi _{1}...d\xi _{2m} \\
&=&\frac{1}{(2\pi )^{2dm}}(\int_{\mathbb{R}^{2m}}\exp (-\frac{1}{2}%
\sum_{k=1}^{d}\left\langle \xi ^{1},Q(\mathbf{u})\xi ^{1}\right\rangle _{%
\mathbb{R}^{2m}})d\xi ^{1})^{d} \\
&\leq &\frac{1}{(2\pi )^{dm}}(\det Q(\mathbf{u}))^{-\frac{d}{2}}.
\end{eqnarray*}%
Using the last estimate, we get that%
\begin{eqnarray*}
&&\mathbb{E}[(\int_{0}^{T}(K_{H}^{-1}(\int_{0}^{.}\varphi _{x,\varepsilon
}(B_{u}^{H})du)(t))^{2}dt)^{m}] \\
&\leq &\frac{T^{m-1}}{(2\pi )^{dm}}T^{2m(\frac{1}{2}-H(1+d))} \\
&&\times (2m)!\int_{\mathcal{T}_{2m}(0,1)}(\prod_{j=1}^{2m}\gamma _{-%
\frac{1}{2}-H,\frac{1}{2}-H}(1,u_{j}))(\det Q(\mathbf{u}))^{-\frac{d}{2}}d%
\mathbf{u} \\
&\leq &\frac{T^{m-1}}{(2\pi )^{dm}}T^{2m(\frac{1}{2}-H(1+d))} \\
&&\times C_{H,d}^{m}(m!)^{2H(1+d)},
\end{eqnarray*}%
where the last bound is due to Lemma A.5 for a constant $C_{H,d}$ only
depending on $H$ and $d$. So the result follows.

\end{proof}

\begin{prop}\label{weaksolution}
Let $x\in \mathbb{R}^{d}$ and $H<\frac{1}{2(1+d)}$. Then there exists a $%
X\in L^{p}(\Omega )$ such that 
\begin{equation*}
\mathcal{E}(\int_{0}^{T}K_{H}^{-1}(\int_{0}^{\cdot }\varphi
_{x,1/n}(B_{u}^{H})\boldsymbol{1}_{d}du)^{\ast }(s)dW_{s})\underset{%
n\longrightarrow \infty }{\longrightarrow }X\text{ in }L^{p}(\Omega )
\end{equation*}%
for all $p\geq 1$. Furthermore,%
\begin{equation*}
B_{t}^{H}-L_{t}^{x}(B^{H})\boldsymbol{1}_{d},0\leq t\leq T
\end{equation*}%
is a fractional Brownian motion with Hurst parameter $H$ under the change of
measure with respect to the Radon-Nikodym-derivative $X$.
\end{prop}
\begin{proof}
Without loss of generality let $p=1$. Then using $\left\vert
e^{x}-e^{y}\right\vert \leq \left\vert x-y\right\vert e^{x+y}$, H\"{o}lder's
inequality, the supermartingale property of Doleans-Dade exponentials we get
in connection with the previous Lemma that%
\begin{eqnarray*}
&&E[\left\vert \mathcal{E}(\int_{0}^{T}K_{H}^{-1}(\int_{0}^{\cdot }\varphi
_{x,1/n}(B_{u}^{H})\boldsymbol{1}_{d}du)^{\ast }(s)dW_{s})-\mathcal{E}%
(\int_{0}^{T}K_{H}^{-1}(\int_{0}^{\cdot }\varphi _{x,1/r}(B_{u}^{H})%
\boldsymbol{1}_{d}du)^{\ast }(s)dW_{s})\right\vert ] \\
&\leq &C(I_{1}+I_{2})E,
\end{eqnarray*}%
where%
\begin{equation*}
I_{1}:=E[\int_{0}^{T}\left\vert K_{H}^{-1}(\int_{0}^{\cdot }\varphi
_{1/n}(B_{u}^{H})\boldsymbol{1}_{d}du)^{\ast }(s)-K_{H}^{-1}(\int_{0}^{\cdot
}\varphi _{1/r}(B_{u}^{H})\boldsymbol{1}_{d}du)^{\ast }(s)\right\vert
^{2}ds]^{1/2},
\end{equation*}%
\begin{equation*}
I_{2}:=E[(\int_{0}^{T}\left\vert K_{H}^{-1}(\int_{0}^{\cdot }\varphi
_{1/n}(B_{u}^{H})\boldsymbol{1}_{d}du)^{\ast }(s)\right\vert
^{2}ds-\int_{0}^{T}\left\vert K_{H}^{-1}(\int_{0}^{\cdot }\varphi
_{1/r}(B_{u}^{H})\boldsymbol{1}_{d}du)^{\ast }(s)\right\vert
^{2}ds)^{2}]^{1/2}
\end{equation*}%
\begin{eqnarray*}
E &:&=E[\exp \{\mu _{1}\int_{0}^{t}\left\vert K_{H}^{-1}(\int_{0}^{\cdot
}\varphi _{1/n}(B_{u}^{H})\boldsymbol{1}_{d}du)^{\ast }(s)\right\vert
^{2}ds\}]^{1/4} \\
&&\cdot E[\exp \{\mu _{2}\int_{0}^{t}\left\vert K_{H}^{-1}(\int_{0}^{\cdot
}\varphi _{1/r}(B_{u}^{H})\boldsymbol{1}_{d}du)^{\ast }(s)\right\vert
^{2}ds\}]^{1/4}
\end{eqnarray*}%
for constants $C,\mu _{1},\mu _{2}>0$.

Now, let us have a look at the proof of the previous Lemma and adopt the
notation therein. In the sequel we omit $\boldsymbol{1}_{d}$. Then we obtain
for $m=1$ by using the self-similarity of the fBm in a similar way (but
under expectation) that%
\begin{eqnarray*}
&&E[\left\vert K_{H}^{-1}(\int_{0}^{\cdot }\varphi _{\varepsilon
_{1}}(B_{u}^{H})du)^{\ast }(t)\right\vert ^{2}\left\vert
K_{H}^{-1}(\int_{0}^{\cdot }\varphi _{\varepsilon _{2}}(B_{u}^{H})du)^{\ast
}(t)\right\vert ^{2}] \\
&=&E[(t^{2m(\frac{1}{2}-H(1+d))}(2m)!)^{2}\int_{\mathcal{T}_{2m}(0,1)}\prod_{j=1}^{2m}\gamma _{-\frac{1}{2}-H,\frac{1}{2}-H}(1,u_{j})\varphi
_{xt^{-H},\varepsilon _{1}(t)}(B_{u_{j}}^{H})d\mathbf{u} \\
&&\times \int_{\mathcal{T}_{2m}(0,1)}\prod_{j=1}^{2m}\gamma _{-\frac{%
1}{2}-H,\frac{1}{2}-H}(1,u_{j})\varphi _{xt^{-H},\varepsilon
_{2}(t)}(B_{u_{j}}^{H})d\mathbf{u],}
\end{eqnarray*}%
where $\varepsilon _{i}(t)=\varepsilon _{i}t^{-2H},$ $i=1,2.$ Using
shuffling (see Section 2.2), we get that%
\begin{eqnarray*}
&&E[\left\vert K_{H}^{-1}(\int_{0}^{\cdot }\varphi _{\varepsilon
_{1}}(B_{u}^{H})du)^{\ast }(t)\right\vert ^{2}\left\vert
K_{H}^{-1}(\int_{0}^{\cdot }\varphi _{\varepsilon _{2}}(B_{u}^{H})du)^{\ast
}(t)\right\vert ^{2}] \\
&=&E[(t^{2m(\frac{1}{2}-H(1+d))}(2m)!)^{2} \\
&&\times \sum_{\sigma \in S(2m,2m)}\int_{\mathcal{T}_{4m}(0,1)}\prod_{j=1}^{4m}f_{\sigma (j)}(u_{j})d\mathbf{u}],
\end{eqnarray*}%
where $f_{j}(s):=\gamma _{-\frac{1}{2}-H,\frac{1}{2}-H}(1,s)\varphi
_{xt^{-H},\varepsilon _{1}(t)}(B_{s}^{H}),$ if $j=1,...,2m$ and $\gamma _{-%
\frac{1}{2}-H,\frac{1}{2}-H}(1,s)\varphi _{xt^{-H},\varepsilon
_{2}(t)}(B_{s}^{H}),$ if $j=2m+1,...,4m$. Without loss of generality,
consider the case  
\begin{eqnarray*}
&&\prod_{j=1}^{4m}f_{\sigma (j)}(u_{j}) \\
&=&\prod_{j=1}^{2m}\gamma _{-\frac{1}{2}-H,\frac{1}{2}%
-H}(1,u_{j})\varphi _{xt^{-H},\varepsilon _{1}(t)}(B_{u_{j}}^{H}) \\
&&\times \prod_{j=2m+1}^{4m}\gamma _{-\frac{1}{2}-H,\frac{1}{2}%
-H}(1,u_{j})\varphi _{xt^{-H},\varepsilon _{2}(t)}(B_{u_{j}}^{H}).
\end{eqnarray*}%
Then%
\begin{eqnarray*}
&&E[(t^{2m(\frac{1}{2}-H(1+d))}(2m)!)^{2}\int_{\mathcal{T}_{4m}(0,1)}\prod_{j=1}^{4m}f_{\sigma (j)}(u_{j})d\mathbf{u}] \\
&=&(t^{2m(\frac{1}{2}-H(1+d))}(2m)!)^{2} \\
&&\times \int_{\mathcal{T}_{4m}(0,1)}\prod_{j=1}^{2m}\gamma _{-\frac{%
1}{2}-H,\frac{1}{2}-H}(1,u_{j})\prod_{j=2m+1}^{4m}\gamma _{-\frac{1}{%
2}-H,\frac{1}{2}-H}(1,u_{j}) \\
&&\times E[\prod_{j=1}^{2m}\varphi _{xt^{-H},\varepsilon
_{1}(t)}(B_{u_{j}}^{H})\prod_{j=2m+1}^{4m}\varphi
_{xt^{-H},\varepsilon _{2}(t)}(B_{u_{j}}^{H})]d\mathbf{u} \\
&=&(t^{2m(\frac{1}{2}-H(1+d))}(2m)!)^{2}\int_{\mathcal{T}_{4m}(0,1)}\prod_{j=1}^{4m}\gamma _{-\frac{1}{2}-H,\frac{1}{2}-H}(1,u_{j}) \\
&&\times E[\prod_{j=1}^{2m}\int_{\mathbb{R}^{d}}\exp (i\left\langle
\xi _{j},B_{u_{j}}^{H}-xt^{-H}\right\rangle _{\mathbb{R}^{d}}-\varepsilon
_{1}(t)\frac{\left\vert \xi _{j}\right\vert _{\mathbb{R}^{d}}^{2}}{2})d\xi
_{j} \\
&&\times \prod_{j=2m+1}^{4m}\int_{\mathbb{R}^{d}}\exp (i\left\langle
\xi _{j},B_{u_{j}}^{H}-xt^{-H}\right\rangle _{\mathbb{R}^{d}}-\varepsilon
_{2}(t)\frac{\left\vert \xi _{j}\right\vert _{\mathbb{R}^{d}}^{2}}{2})d\xi
_{j}]d\mathbf{u} \\
&=&(t^{2m(\frac{1}{2}-H(1+d))}(2m)!)^{2} \\
&&\times \int_{\mathcal{T}_{4m}(0,1)}\prod_{j=1}^{4m}\gamma _{-\frac{%
1}{2}-H,\frac{1}{2}-H}(1,u_{j}) \\
&&\times \frac{1}{(2\pi )^{4dm}}\int_{\mathbb{R}^{4dm}}E[\exp
(i\sum_{j=1}^{4m}\left\langle \xi _{j},B_{u_{j}}^{H}\right\rangle _{\mathbb{R%
}^{d}})] \\
&&\times \exp (-i\sum_{j=1}^{4m}\left\langle \xi _{j},xt^{-H}\right\rangle _{%
\mathbb{R}^{d}})\exp (-\frac{\varepsilon _{1}(t)}{2}\sum_{j=1}^{2m}\left%
\vert \xi _{j}\right\vert _{\mathbb{R}^{d}}^{2}-\frac{\varepsilon _{2}(t)}{2}%
\sum_{j=1}^{2m}\left\vert \xi _{j}\right\vert _{\mathbb{R}^{d}}^{2}) \\
&&d\xi _{1}...d\xi _{4m}d\mathbf{u}.
\end{eqnarray*}%
So%
\begin{eqnarray*}
&&E[(t^{2m(\frac{1}{2}-H(1+d))}(2m)!)^{2}\int_{\mathcal{T}_{4m}(0,1)}\prod_{j=1}^{4m}f_{\sigma (j)}(u_{j})d\mathbf{u}] \\
&=&(t^{2m(\frac{1}{2}-H(1+d))}(2m)!)^{2}\int_{\mathcal{T}_{4m}(0,1)}\prod_{j=1}^{4m}\gamma _{-\frac{1}{2}-H,\frac{1}{2}-H}(1,u_{j}) \\
&&\times \frac{1}{(2\pi )^{4dm}}\int_{\mathbb{R}^{4dm}}\exp (-\frac{1}{2}%
\sum_{k=1}^{d}\left\langle \xi ^{k},Q(\mathbf{u})\xi ^{k}\right\rangle _{%
\mathbb{R}^{4m}}) \\
&&\times \exp (-i\sum_{j=1}^{4m}\left\langle \xi _{j},xt^{-H}\right\rangle _{%
\mathbb{R}^{d}})\exp (-\frac{\varepsilon _{1}(t)}{2}\sum_{j=1}^{2m}\left%
\vert \xi _{j}\right\vert _{\mathbb{R}^{d}}^{2}-\frac{\varepsilon _{2}(t)}{2}%
\sum_{j=1}^{2m}\left\vert \xi _{j}\right\vert _{\mathbb{R}^{d}}^{2}) \\
&&d\xi _{1}...d\xi _{4m}d\mathbf{u}.
\end{eqnarray*}%
Hence, using dominated convergence in connection with Lemma A.5, we see that%
\begin{eqnarray*}
&&\int_{0}^{T}E[(t^{2m(\frac{1}{2}-H(1+d))}(2m)!)^{2}\int_{\mathcal{T}%
_{4m}(0,1)}\prod_{j=1}^{4m}f_{\sigma (j)}(u_{j})d\mathbf{u}]dt \\
&\longrightarrow &\int_{0}^{T}(t^{2m(\frac{1}{2}-H(1+d))}(2m)!)^{2}\int_{\mathcal{T}%
_{4m}(0,1)}\prod_{j=1}^{4m}\gamma _{-\frac{1}{2}-H,\frac{1}{2}%
-H}(1,u_{j}) \\
&&\times \frac{1}{(2\pi )^{4dm}}\int_{\mathbb{R}^{4dm}}\exp (-\frac{1}{2}%
\sum_{k=1}^{d}\left\langle \xi ^{k},Q(\mathbf{u})\xi ^{k}\right\rangle _{%
\mathbb{R}^{4m}}) \\
&&\exp (-i\sum_{j=1}^{4m}\left\langle \xi _{j},xt^{-H}\right\rangle _{%
\mathbb{R}^{d}})d\xi _{1}...d\xi _{4m}d\mathbf{u}dt
\end{eqnarray*}%
for $\varepsilon _{1},\varepsilon _{2}\searrow 0$. For other $\sigma \in
S(2m,2m)$, we obtain similar limit values. In summary, we find (by also
considering the case $\varepsilon _{1}=\varepsilon _{2}$) that%

\begin{equation*}
E[(\int_{0}^{T}\left\vert K_{H}^{-1}(\int_{0}^{\cdot }\varphi _{\varepsilon
_{1}}(B_{u}^{H})du)(s)-K_{H}^{-1}(\int_{0}^{\cdot }\varphi _{\varepsilon
_{2}}(B_{u}^{H})du)(s)\right\vert ^{2}ds)^{2}]\longrightarrow 0
\end{equation*}%
for $\varepsilon _{1},\varepsilon _{2}\searrow 0$.
Thus%
\begin{equation*}
I_{2}=I_{2}(n,r)\longrightarrow 0\text{ for }n,r\longrightarrow \infty .%
\text{ }
\end{equation*}%
Similarly, we have that%
\begin{equation*}
I_{1}=I_{1}(n,r)\longrightarrow 0\text{ for }n,r\longrightarrow \infty .
\end{equation*}%
Since $E=E(n,r)$ is uniformly bounded with respect to $n,r$ because of the
previous Lemma \ref{expmom}, we obtain the convergence of the Radon-Nikodym derivatives
to a $X$ in $L^{p}(\Omega )$ for $p=1$. The second statement of the Lemma
follows by using characteristic functions combined with dominated
convergence.
\end{proof}

Henceforth, we confine ourselves to the filtered probability space
$(\Omega,\mathfrak{A},P)$, $\mathcal{F}=\{\mathcal{F}_{t}\}_{t\in[0,T]}$
which carries the weak solution $(X^{x},B^{H})$ of \eqref{SDE}.

We now turn to the second step of our procedure.
\begin{lem}\label{weakconv}
Suppose that $H<\frac{1}{2(1+d)}$ and let $\{\varphi
_{\varepsilon }\}_{\varepsilon >0}$ be defined as%
\begin{equation*}
\varphi _{\varepsilon }(y)=\varphi _{\varepsilon ,x}(y)=\varepsilon ^{-\frac{%
d}{2}}\varphi (\varepsilon ^{-\frac{1}{2}}(y-x)),\varepsilon >0,
\end{equation*}%
where $\varphi $ is the $d-$dimensional standard normal density. Denote by $%
X^{x,\varepsilon }=\{X_{t}^{x,\varepsilon },t\in \lbrack 0,T]\}$ the
corresponding solutions of (\ref{SDE}), if we replace $\delta _{x}$ by $\varphi
_{\varepsilon ,x}(y),\varepsilon >0$. Then for every $t\in \lbrack 0,T]$ and
bounded continuous function $\eta :\mathbb{R}^{d}\longrightarrow \mathbb{R}$
we have that%
\begin{equation*}
\eta (X_{t}^{x,\varepsilon })\overset{\varepsilon \longrightarrow 0_{+}}{%
\longrightarrow }E[\eta (X_{t}^{x})\left\vert \mathcal{F}_{t}\right]
\end{equation*}%
weakly in $L^{2}(\Omega ,\mathcal{F}_{t},P)$.

\end{lem}
\begin{proof}
Without loss of generality let $x=0$. We mention that%
\begin{equation*}
\Sigma _{t}:=\{\exp \{\sum_{j=1}^{k}\left\langle \alpha
_{j},B_{t_{j}}^{H}-B_{t_{j-1}}^{H}\right\rangle \}:\{\alpha
_{j}\}_{j=1}^{k}\subset \mathbb{R}^{d},0=t_{0}<...<t_{k}=t,k\geq 1\}
\end{equation*}%
is a total subset of $L^{2}(\Omega ,\mathcal{F}_{t},P).$ Denote $X_{t}^{x,\varepsilon }$ by $X_{t}^{n}$ for $\varepsilon=1/n$. Then using Girsanov's
theorem, we find that%
\begin{eqnarray*}
&&E[\eta (X_{t}^{n})\exp \{\sum_{j=1}^{k}\left\langle \alpha
_{j},B_{t_{j}}^{H}-B_{t_{j-1}}^{H}\right\rangle \}] \\
&=&E[\eta (X_{t}^{n})\exp \{\sum_{j=1}^{k}\left\langle \alpha
_{j},X_{t_{j}}^{n}-X_{t_{j-1}}^{n}-\int_{t_{j-1}}^{t_{j}}\varphi
_{1/n}(X_{s}^{n})\boldsymbol{1}_{d}ds\right\rangle \}] \\
&=&E[\eta (B_{t}^{H})\exp \{\sum_{j=1}^{k}\left\langle \alpha
_{j},B_{t_{j}}^{H}-B_{t_{j-1}}^{H}-\int_{t_{j-1}}^{t_{j}}\varphi
_{1/n}(B_{s}^{H})\boldsymbol{1}_{d}ds\right\rangle \} \\
&&\cdot \mathcal{E}(\int_{0}^{t}K_{H}^{-1}(\int_{0}^{\cdot }\varphi
_{1/n}(B_{u}^{H})\boldsymbol{1}_{d}du)^{\ast }(s)dW_{s})].
\end{eqnarray*}

On the other hand, we obtain by $\left\vert e^{x}-e^{y}\right\vert \leq
\left\vert x-y\right\vert e^{x+y}$, H\"{o}lder's inequality, the
supermartingale property of Doleans-Dade exponentials and the proof of
Proposition \ref{weaksolution} that%
\begin{eqnarray*}
&&\left\vert E[\eta (B_{t}^{H})\exp \{\sum_{j=1}^{k}\left\langle \alpha
_{j},B_{t_{j}}^{H}-B_{t_{j-1}}^{H}-\int_{t_{j-1}}^{t_{j}}\varphi
_{1/n}(B_{s}^{H})\boldsymbol{1}_{d}ds\right\rangle \}\right.  \\
&&\left. \cdot \mathcal{E}(\int_{0}^{t}K_{H}^{-1}(\int_{0}^{\cdot }\varphi
_{1/n}(B_{u}^{H})\boldsymbol{1}_{d}du)^{\ast }(s)dW_{s})\right.  \\
&&\left. -\eta (B_{t}^{H})\exp \{\sum_{j=1}^{k}\left\langle \alpha
_{j},B_{t_{j}}^{H}-B_{t_{j-1}}^{H}-\int_{t_{j-1}}^{t_{j}}\delta
_{0}(B_{s}^{H})ds\boldsymbol{1}_{d}\right\rangle \}\right.  \\
&&\left. \cdot X]\right\vert  \\
&\leq &C(I_{1}+I_{2}+I_{3})E,
\end{eqnarray*}%

where 
\begin{equation*}
I_{1}:=E[\left( \sum_{j=1}^{k}\left\langle \alpha
_{j},\int_{t_{j-1}}^{t_{j}}\delta _{0}(B_{s}^{H})ds\boldsymbol{1}%
_{d}-\int_{t_{j-1}}^{t_{j}}\varphi _{1/n}(B_{s}^{H})\boldsymbol{1}%
_{d}ds\right\rangle \right) ^{2}]^{1/2},
\end{equation*}%
\begin{equation*}
I_{2}:=\lim_{r\longrightarrow \infty }E[\int_{0}^{t}\left\vert
K_{H}^{-1}(\int_{0}^{\cdot }\varphi _{1/n}(B_{u}^{H})\boldsymbol{1}%
_{d}du)^{\ast }(s)-K_{H}^{-1}(\int_{0}^{\cdot }\varphi _{1/r}(B_{u}^{H})%
\boldsymbol{1}_{d}du)^{\ast }(s)\right\vert ^{2}ds]^{1/2},
\end{equation*}%
\begin{equation*}
I_{3}:=\lim_{r\longrightarrow \infty }E[(\int_{0}^{t}\left\vert
K_{H}^{-1}(\int_{0}^{\cdot }\varphi _{1/n}(B_{u}^{H})\boldsymbol{1}%
_{d}du)^{\ast }(s)\right\vert ^{2}ds-\int_{0}^{t}\left\vert
K_{H}^{-1}(\int_{0}^{\cdot }\varphi _{1/r}(B_{u}^{H})\boldsymbol{1}%
_{d}du)^{\ast }(s)\right\vert ^{2}ds)^{2}]^{1/2}
\end{equation*}%
and%
\begin{align*}
E &:=\sup_{r\geq 1}E[\exp \{8\sum_{j=1}^{k}\left\langle \alpha_{j},B_{t_{j}}^{H}-B_{t_{j-1}}^{H}-\int_{t_{j-1}}^{t_{j}}\varphi_{1/n}(B_{s}^{H})\boldsymbol{1}_{d}ds\right\rangle \}]^{1/8} \\
&\cdot E[\exp \{8\sum_{j=1}^{k}\left\langle \alpha_{j},B_{t_{j}}^{H}-B_{t_{j-1}}^{H}-\int_{t_{j-1}}^{t_{j}}\delta_{0}(B_{s}^{H})ds\boldsymbol{1}_{d}\right\rangle \}]^{1/8} \\
&\cdot E[\exp \{\mu _{1}\int_{0}^{t}\left\vert K_{H}^{-1}(\int_{0}^{\cdot
}\varphi _{1/n}(B_{u}^{H})\boldsymbol{1}_{d}du)^{\ast }(s)\right\vert
^{2}ds\}]^{1/8} \\
&\cdot E[\exp \{\mu _{2}\int_{0}^{t}\left\vert K_{H}^{-1}(\int_{0}^{\cdot
}\varphi _{1/r}(B_{u}^{H})ds\boldsymbol{1}_{d}du)^{\ast }(s)\right\vert^{2}ds\}]^{1/16}
\end{align*}
for constants $C,\mu _{1},\mu _{2}>0$.

By inspecting the proof of Proposition \ref{weaksolution} once again, we know that%
\begin{equation*}
I_{3}=I_{3}(n)\longrightarrow 0\text{ for }n\longrightarrow \infty .\text{ }
\end{equation*}%
and%
\begin{equation*}
I_{2}=I_{2}(n)\longrightarrow 0\text{ for }n\longrightarrow \infty .
\end{equation*}%
Since $L_{t}^{x}(B^{H},\varepsilon )$ converges to $L_{t}^{x}(B^{H})$ in $%
L^{p}(\Omega )$ for all $p\geq 1$, we also conclude that%
\begin{equation*}
I_{1}=I_{1}(n)\longrightarrow 0\text{ for }n\longrightarrow \infty .
\end{equation*}

On the other hand, we obtain from (\ref{eq:BoundMomentLT}), Theorem \ref{thm:LDforfBM} and Lemma \ref{expmom} that%
\begin{equation*}
E=E(n)\leq K
\end{equation*}%
for all $n$, where $K$ is a constant. So we see that%
\begin{eqnarray*}
&&E[\eta (X_{t}^{n})\exp \{\sum_{j=1}^{k}\left\langle \alpha
_{j},B_{t_{j}}^{H}-B_{t_{j-1}}^{H}\right\rangle \}]\longrightarrow \\
&&E[\eta (B_{t}^{H})\exp \{\sum_{j=1}^{k}\left\langle \alpha
_{j},B_{t_{j}}^{H}-B_{t_{j-1}}^{H}-\int_{t_{j-1}}^{t_{j}}\delta
_{0}(B_{s}^{H})\boldsymbol{1}_{d}ds\right\rangle \} \\
&&\cdot X] \\
&=&E[E[\eta (X_{t})\left\vert \mathcal{F}_{t}\right] \exp
\{\sum_{j=1}^{k}\left\langle \alpha
_{j},B_{t_{j}}^{H}-B_{t_{j-1}}^{H}\right\rangle \}]
\end{eqnarray*}%
for $n\longrightarrow \infty ,$ which completes the proof.
\end{proof}

We continue with the third step of our scheme. This is the most challenging part. 
For notational convenience let us \emph{from know on} assume that $\alpha =1$
in (\ref{SDE}) and that $\varphi _{\varepsilon }^{\shortmid }$ stands for the Jacobian of $\varphi
_{\varepsilon }\mathbf{1}_{d}$.
The following result is based on a compactness criterion for subsets of $L^2(\Omega)$ which is summarised in the Appendix.

\begin{lem}\label{relcomp}
Let $\{\varphi_{\varepsilon}\}_{\varepsilon > 0}$ the family of Gaussian kernels approximating Dirac's delta function $\delta_0$ in the sense of \eqref{Xn}. Fix $t\in [0,T]$ and denote by $X_t^{\varepsilon}$ the corresponding solutions of \eqref{SDE} if we replace $L_t(X^x)$ by $\int_0^t \varphi_{\varepsilon}(X_s^{\varepsilon}) ds$, $\varepsilon > 0$. Then there exists a $\beta\in (0,1/2)$ such that
$$\sup_{\varepsilon > 0} \int_0^t \int_0^t \frac{E[\|D_\theta X_t^{\varepsilon} - D_{\theta'} X_t^{\varepsilon}\|^2]}{|\theta' - \theta|^{1+2\beta}} d\theta' d\theta <\infty$$
and
\begin{align}\label{Mallest}
\sup_{\varepsilon > 0} \|D_{\cdot} X_t^{\varepsilon}\|_{L^2(\Omega\times [0,T],\R^{d\times d})} <\infty.
\end{align}
\end{lem}
\begin{proof}
Fix $t\in [0,T]$ and take $\theta,\theta'>0$ such that $0<\theta'<\theta < t$. Using the chain rule for the Malliavin derivative, see \cite[Proposition 1.2.3]{Nua10}, we have
$$D_{\theta} X_t^{\varepsilon} = K_H(t,\theta) I_{d} + \int_{\theta}^t \varphi_{\varepsilon}'(X_s^{\varepsilon}) D_{\theta} X_s^{\varepsilon} ds$$
$P$-a.s. for all $0\leq \theta \leq t$ where $\varphi_{\varepsilon}'(z) = \left(\frac{\partial}{\partial z_j} \varphi_{\varepsilon}^{(i)} (z) \right)_{i,j=1,\dots, d}$ denotes the Jacobian matrix of $\varphi_{\varepsilon}$ and $I_d$ the identity matrix in $\R^{d\times d}$. Thus we have

\begin{align*}
D_{\theta'} X_t^{\varepsilon} -& D_{\theta} X_t^{\varepsilon} = K_H(t,\theta')I_d- K_H(t,\theta)I_d\\
&+\int_{\theta'}^t \varphi_{\varepsilon}'(X_s^{\varepsilon}) D_{\theta'} X_s^{\varepsilon} ds - \int_{\theta}^t \varphi_{\varepsilon}'(X_s^{\varepsilon}) D_{\theta} X_s^{\varepsilon} ds\\
=& K_H(t,\theta')I_d- K_H(t,\theta) I_d\\
&+\int_{\theta'}^{\theta} \varphi_{\varepsilon}'(X_s^{\varepsilon}) D_{\theta'} X_s^{\varepsilon} ds+\int_{\theta}^t  \varphi_{\varepsilon}'(X_s^{n}) (D_{\theta'} X_s^{\varepsilon} - D_{\theta}X_s^{\varepsilon}) ds\\
=& K_H(t,\theta')I_d - K_H(t,\theta) I_d + D_{\theta'}X_{\theta}^{\varepsilon} - K_H (\theta,\theta')I_d  \\
&+ \int_{\theta}^t \varphi_{\varepsilon}'(X_s^{\varepsilon})(D_{\theta'} X_s^{\varepsilon} - D_{\theta}X_s^{\varepsilon})ds.
\end{align*}
Using Picard iteration applied to the above equation we may write

\begin{align*}
D_{\theta'} X_t^{\varepsilon} -& D_{\theta} X_t^{\varepsilon} = K_H(t,\theta')I_d - K_H (t,\theta)I_d\\
&+ \sum_{m=1}^{\infty} \int_{\Delta_{\theta,t}^m} \prod_{j=1}^m  \varphi_{\varepsilon}'(X_{s_j}^{\varepsilon}) \left(K_H(s_m,\theta')I_d - K_H (s_m,\theta)I_d\right) ds_m \cdots ds_1\\
&+ \left( I_d + \sum_{m=1}^{\infty} \int_{\Delta_{\theta,t}^m} \prod_{j=1}^m  \varphi_{\varepsilon}'(X_{s_j}^{\varepsilon})ds_m \cdots ds_1 \right) \left(D_{\theta'} X_{\theta}^{\varepsilon} - K_H(\theta,\theta') I_d\right).
\end{align*}
On the other hand, observe that one may again write
\begin{align*}
D_{\theta'} X_{\theta}^{\varepsilon} - K_H(\theta,\theta')I_d = \sum_{m=1}^{\infty} \int_{\Delta_{\theta',\theta}^m} \prod_{j=1}^m  \varphi_{\varepsilon}'(X_{s_j}^{\varepsilon}) (K_H(s_m,\theta') I_d) \, ds_m \cdots ds_1.
\end{align*}
Altogether, we can write
$$D_{\theta'} X_t^{\varepsilon} - D_{\theta} X_t^{\varepsilon} = I_1(\theta',\theta) + I_2^{\varepsilon} (\theta',\theta)+ I_3^{\varepsilon} (\theta',\theta),$$
where
\begin{align*}
I_1(\theta',\theta) :=&  K_H(t,\theta')I_d - K_H (t,\theta)I_d\\
I_2^{\varepsilon}(\theta',\theta) :=& \sum_{m=1}^{\infty} \int_{\Delta_{\theta,t}^m} \prod_{j=1}^m  \varphi_{\varepsilon}'(X_{s_j}^{\varepsilon}) \left( K_H(s_m,\theta')I_d - K_H (s_m,\theta)I_d \right) ds_m \cdots ds_1\\
I_3^{\varepsilon}(\theta',\theta) :=&\left(I_d+ \sum_{m=1}^{\infty} \int_{\Delta_{\theta,t}^m} \prod_{j=1}^m  \varphi_{\varepsilon}'(X_{s_j}^{\varepsilon})ds_m \cdots ds_1 \right)\\
&\times\left(\sum_{m=1}^{\infty} \int_{\Delta_{\theta',\theta}^m} \prod_{j=1}^m  \varphi_{\varepsilon}'(X_{s_j}^{\varepsilon}) (K_H(s_m,\theta')I_d) ds_m \cdots ds_1.\right).
\end{align*}

It follows from Lemma \ref{doubleint} that
\begin{align}\label{double1}
\int_0^t \int_0^t \frac{\|I_1(\theta',\theta)\|_{L^2(\Omega)}^2}{|\theta'-\theta|^{1+2\beta}}d\theta d\theta' = \int_0^t \int_0^t \frac{|K_H(t,\theta')-K_H(t,\theta)|^2}{|\theta'-\theta|^{1+2\beta}}d\theta d\theta'<\infty
\end{align}
for a suitably small $\beta \in (0,1/2)$.

Let us continue with the term $I_2^n(\theta',\theta)$. Then Girsanov's theorem, Cauchy-Schwarz inequality and Lemma \ref{expmom} imply
\begin{align*}
E[&\| I_2^{\varepsilon}(\theta',\theta) \|^2]\\
&\leq C E\left[ \left\|\sum_{m=1}^{\infty} \int_{\Delta_{\theta,t}^m} \prod_{j=1}^m  \varphi_{\varepsilon}'(x+B_{s_j}^H) \left( K_H(s_m,\theta')I_d - K_H (s_m,\theta)I_d \right) ds_m \cdots ds_1\right\|^4 \right]^{1/2},
\end{align*}
where $C>0$ is an upperbound from Lemma \ref{expmom}.

Let $\|\cdot\|$ denote the matrix norm in $\R^{d\times d}$ such that $\|A\| = \sum_{i,j=1}^d |a_{ij}|$ for a matrix $A=\{a_{ij}\}_{i,j=1,\dots,d}$, then taking this matrix norm and expectation we have

\begin{align*}
E[\| I_2^{\varepsilon}(\theta',\theta) \|^2] \leq& C\Bigg(\sum_{m=1}^{\infty} \sum_{i,j=1}^d \sum_{l_1,\dots, l_{m-1}=1}^d \Bigg\|\int_{\Delta_{\theta,t}^m} \frac{\partial}{\partial x_{l_1}} \varphi_{\varepsilon}^{(i)} (x+B_{s_1}^H) \frac{\partial}{\partial x_{l_2}}\varphi_{\varepsilon}^{(l_1)} (x+B_{s_2}^H) \cdots\\
& \cdots \frac{\partial}{\partial x_j}\varphi_{\varepsilon}^{(l_{m-1})} (x+B_{s_m}^H) \left( K_H(s_m,\theta') - K_H (s_m,\theta) \right) ds_m \cdots ds_1\Bigg\|_{L^4(\Omega, \R)}\Bigg)^{2}  .
\end{align*}

Now we concentrate on the expression
\begin{align}\label{I}
J_2^{\varepsilon}(\theta',\theta) := \int_{\Delta_{\theta,t}^m} \frac{\partial}{\partial x_{l_1}} \varphi_{\varepsilon}^{(i)} (x+B_{s_1}^H) \cdots \frac{\partial}{\partial x_j}\varphi_{\varepsilon}^{(l_{m-1})} (x+B_{s_m}^H) \left( K_H(s_m,\theta') - K_H (s_m,\theta) \right) ds.
\end{align}
Then, shuffling $J_2^{\varepsilon}(\theta',\theta)$ as shown in \eqref{shuffleIntegral}, one can write $(J_2^{\varepsilon}(\theta',\theta))^2$ as a sum of at most $2^{2m}$ summands of length $2m$ of the form
\begin{align}\label{II}
\int_{\Delta_{\theta,t}^{2m}} g_1^{\varepsilon} (B_{s_1}^H) \cdots g_{2m}^{\varepsilon} (B_{s_{2m}}^H) ds_{2m} \cdots ds_1,
\end{align}
where for each $l=1,\dots, 2m$,
$$g_l^{\varepsilon}(B_{\cdot}^H) \in \left\{ \frac{\partial}{\partial x_j} \varphi_{\varepsilon}^{(i)} (x+B_{\cdot}^H),  \frac{\partial}{\partial x_j} \varphi_{\varepsilon}^{(i)} (x+B_{\cdot}^H)\left( K_H(\cdot,\theta') - K_H (\cdot,\theta) \right), \, i,j=1,\dots,d\right\}.$$

Repeating this argument once again, we find that $J_2^{\varepsilon}(\theta',\theta)^4$ can be expressed as a sum of, at most, $2^{8m}$ summands of length $4m$ of the form
\begin{align}\label{III}
\int_{\Delta_{\theta,t}^{4m}} g_1^{\varepsilon} (B_{s_1}^H) \cdots g_{4m}^{\varepsilon} (B_{s_{4m}}^H) ds_{4m} \cdots ds_1,
\end{align}
where for each $l=1,\dots, 4m$,
$$g_l^{\varepsilon}(B_{\cdot}^H) \in \left\{ \frac{\partial}{\partial x_j} \varphi_{\varepsilon}^{(i)} (x+B_{\cdot}^H),  \frac{\partial}{\partial x_j} \varphi_{\varepsilon}^{(i)} (x+B_{\cdot}^H)\left( K_H(\cdot,\theta') - K_H (\cdot,\theta) \right), \, i,j=1,\dots,d\right\}.$$

It is important to note that the function $\left( K_H(\cdot,\theta') - K_H (\cdot,\theta) \right)$ appears only once in term \eqref{I} and hence only four times in term \eqref{III}. So there are indices $j_1,\dots, j_4 \in \{1,\dots, 4m\}$ such that we can write \eqref{III} as
$$\int_{\Delta_{\theta,t}^{4m}} \left(\prod_{j=1}^{4m} g_j^{\varepsilon}(B_{s_j}^H)\right)  \prod_{i=1}^4 \left( K_H(s_{j_i},\theta') - K_H (s_{j_i},\theta)\right) ds_{4m} \cdots ds_1,$$
where
$$g_l^{\varepsilon}(B_{\cdot}^H) \in \left\{ \frac{\partial}{\partial x_j} \varphi_{\varepsilon}^{(i)} (x+B_{\cdot}^H), \, i,j=1,\dots,d\right\}, \quad l=1,\dots,4m.$$

The latter enables us to use the estimate from Proposition \ref{mainestimate1} with $\sum_{j=1}^{4m} \varepsilon_j=4$, $\sum_{l=1}^{d}\alpha _{\lbrack \sigma (j)]}^{(1)}=1$ for all $j$, $\left\vert \alpha \right\vert =4m$ and Remark \ref{Remark 3.4}. Thus we obtain that
\begin{align*}
E(J_2^{\varepsilon}(\theta',\theta))^4 \leq   \left(\frac{\theta-\theta'}{\theta \theta'}\right)^{4\gamma} \theta^{4\left(H-\frac{1}{2}-\gamma\right)} C^{4m}  \|\varphi_{\varepsilon}\|_{L^{1}(\R^d)}^{4m} A_m^{\gamma}(H,d, |t-\theta|) 
\end{align*}
whenever $H<\frac{1}{2(2+d)}$ and $\gamma\in (0,H)$, where
\begin{equation*} 
A^{\gamma }(H,d,\left\vert t-\theta \right\vert ):=\frac{((8m)!)^{1/4}(t-%
\theta )^{-H(4m(d+2))-4(H-\frac{1}{2}-\gamma )+4m}}{\Gamma (-H(d+2)8m+8(H-%
\frac{1}{2}-\gamma )+8m)^{1/2}}. \label{Agamma}
\end{equation*}
Note that
$\left\Vert \varphi _{\varepsilon }\right\Vert _{L^{1}(\mathbb{R}^{d})}=1$.

Altogether, we see that
\begin{align*}
E\left[ \|I_2^{\varepsilon} (\theta',\theta)\|^2 \right] \leq \left(\frac{\theta-\theta'}{\theta \theta'}\right)^{2\gamma} \theta^{2\left(H-\frac{1}{2}-\gamma\right)} \left(\sum_{m=1}^\infty d^{m+1} C^m \|\varphi_{\varepsilon}\|_{L^{1}(\R^d)}^m A_m^{\gamma}(H,d,|T|)^{1/4}   \right)^2.
\end{align*}

So we can find a constant $C>0$ such that
\begin{align*}
\sup_{\varepsilon > 0}E\left[ \|I_2^{\varepsilon} (\theta',\theta)\|^2 \right] \leq C\left(\frac{\theta-\theta'}{\theta \theta'}\right)^{2\gamma} \theta^{2\left(H-\frac{1}{2}-\gamma\right)}
\end{align*}
for $\gamma\in (0,H)$ provided that $H<\frac{1}{2(2+d)}$. It is easy to see that we can choose $\gamma \in (0,H)$ such that there is a suitably small $\beta\in (0,1/2)$, $0<\beta< \gamma<H<1/2$ so that
it follows from the proof of Lemma \ref{doubleint} that
\begin{align}\label{double2}
\int_0^t \int_0^t \left|\frac{\theta-\theta'}{\theta \theta'}\right|^{2\gamma} |\theta|^{2\left(H-\frac{1}{2}-\gamma\right)} |\theta-\theta'|^{-1-2\beta} d\theta' d\theta <\infty,
\end{align}
for every $t\in (0,T]$.

We now turn to the term $I_3^{\varepsilon}(\theta',\theta)$. Observe that term $I_3^{\varepsilon} (\theta',\theta)$ is the product of two terms, where the first one will simply be bounded uniformly in $\theta,t \in [0,T]$ under expectation. This can be shown by following meticulously the same steps as we did for $I_2^{\varepsilon}(\theta',\theta)$ and observing that in virtue of Proposition \ref{mainestimate2} with $\varepsilon_j = 0$ for all $j$ the singularity in $\theta$ vanishes.

Again Girsanov's theorem, Cauchy-Schwarz inequality several times and Lemma \ref{expmom} lead to
\begin{align*}
E [\|I_3^{\varepsilon}(\theta',\theta)\|^2] \leq& C \left\|I_d+ \sum_{m=1}^{\infty} \int_{\Delta_{\theta,t}^m} \prod_{j=1}^m  \varphi_{\varepsilon}'(x+B_{s_j}^H)ds_m \cdots ds_1 \right\|_{L^8(\Omega, \R^{d\times d}) }^2 \\
&\times \left\| \sum_{m=1}^{\infty} \int_{\Delta_{\theta',\theta}^m} \prod_{j=1}^m  \varphi_{\varepsilon}'(x+B_{s_j}^H) K_H(s_m,\theta') ds_m \cdots ds_1\right\|_{L^4(\Omega, \R^{d\times d}) }^2,
\end{align*}
where $C>0$ denotes an upperbound obtained from Lemma \ref{expmom}.

Again, we have
\begin{align*}
E [\|I_3^{\varepsilon}(\theta',\theta)\|^2] \leq & C \Bigg(1+ \sum_{m=1}^{\infty} \sum_{i,j=1}^d \sum_{l_1,\dots, l_{m-1}=1}^d \Bigg\|\int_{\Delta_{\theta,t}^m} \frac{\partial}{\partial x_{l_1}} \varphi_{\varepsilon}^{(i)}(x+B_{s_1}^H) \cdots\\
&\cdots  \frac{\partial}{\partial x_j} \varphi_{\varepsilon}^{(l_{m-1})} (x+B_{s_m}^H) ds_m \cdots ds_1 \Bigg\|_{L^8(\Omega, \R)}  \Bigg)^2\\
&\times \Bigg( \sum_{m=1}^{\infty}\sum_{i,j=1}^d \sum_{l_1,\dots, l_{m-1}=1}^d \Bigg\|\int_{\Delta_{\theta',\theta}^m} \frac{\partial}{\partial x_{l_1}} \varphi_{\varepsilon}^{(i)}(x+B_{s_1}^H) \cdots \\
&\cdots  \frac{\partial}{\partial x_j} \varphi_{\varepsilon}^{(l_{m-1})} (x+B_{s_m}^H) K_H(s_m,\theta') ds_m \cdots ds_1\Bigg\|_{L^4(\Omega, \R)} \Bigg)^2.
\end{align*}

Using exactly the same reasoning as for $I_2^{\varepsilon}(\theta',\theta)$ we see that the first factor can be bounded by some finite constant $C$ depending on $H$, $d$, $T$, i.e.
\begin{align*}
E [\|I_3^{\varepsilon}(\theta',\theta)\|^2] \leq& C \Bigg( \sum_{m=1}^{\infty}\sum_{i,j=1}^d \sum_{l_1,\dots, l_{m-1}=1}^d \Bigg\|\int_{\Delta_{\theta',\theta}^m} \frac{\partial}{\partial x_{l_1}} \varphi_{\varepsilon}^{(i)}(x+B_{s_1}^H) \cdots\\
&\cdots \frac{\partial}{\partial x_j} \varphi_{\varepsilon}^{(l_{m-1})} (x+B_{s_m}^H) K_H(s_m,\theta') ds_m \cdots ds_1\Bigg\|_{L^4(\Omega, \R)} \Bigg)^2.
\end{align*}
As before, we pay attention to
\begin{align}\label{IV}
J_3^{\varepsilon}(\theta',\theta) := \int_{\Delta_{\theta',\theta}^m} \frac{\partial}{\partial x_{l_1}} \varphi_{\varepsilon}^{(i)} (x+B_{s_1}^H) \cdots \frac{\partial}{\partial x_j}\varphi_{\varepsilon}^{(l_{m-1})} (x+B_{s_m}^H) K_H(s_m,\theta') ds_m \cdots ds_1.
\end{align}

We can express $(J_3^{\varepsilon}(\theta',\theta))^4$ as a sum of, at most, $2^{8m}$ summands of length $4m$ of the form
\begin{align}\label{V}
\int_{\Delta_{\theta',\theta}^{4m}} g_1^{\varepsilon} (B_{s_1}^H) \cdots g_{4m}^{\varepsilon} (B_{s_{4m}}^H) ds_{4m} \cdots ds_1,
\end{align}
where for each $l=1,\dots, 4m$,
$$g_l^{\varepsilon}(B_{\cdot}^H) \in \left\{ \frac{\partial}{\partial x_j} \varphi_{\varepsilon}^{(i)} (x+B_{\cdot}^H),  \frac{\partial}{\partial x_j} \varphi_{\varepsilon}^{(i)} (x+B_{\cdot}^H) K_H(\cdot,\theta'), \, i,j=1,\dots,d\right\},$$
where the factor $K_H(\cdot,\theta')$ is repeated four times in the integrand of \eqref{V}. Now we can simply apply Proposition \ref{mainestimate2} with $\sum_{j=1}^{4m}\varepsilon_j=4$, $\sum_{l=1}^{d}\alpha _{\lbrack \sigma (j)]}^{(1)}=1$ for all $j$, $\left\vert \alpha \right\vert =4m$ and Remark \ref{Remark 3.4} in order to get
$$E[(J_3^{\varepsilon}(\theta',\theta))^4] \leq \theta^{4\left(H-\frac{1}{2}\right)} C^{4m} \|\varphi_{\varepsilon}\|_{L^{1}(\R^d)}^{4m} A_m^{0}(H,d, |\theta-\theta'|),$$
whenever $H<\frac{1}{2(2+d)}$ where $A_m^{0}(H,d, |\theta-\theta'|)$ is defined as in \eqref{Agamma} by inserting $\gamma =0$.

As a result,
$$E[\|I_3^{\varepsilon}(\theta',\theta)\|^2] \leq \theta^{2\left(H-\frac{1}{2}\right)}\left(\sum_{m=1}^\infty d^{m+1} C^m  \|\varphi_{\varepsilon}\|_{L^{1}(\R^d)}^{m} A_m^0(H,d,|\theta-\theta'|)^{1/4}\right)^2.$$

Since the exponent of $|\theta-\theta'|$ appearing in $A_m^0(H,d,|\theta-\theta|)$ is strictly positive by assumption, we can find a small enough $\delta>0$ and a constant $C:=C_{H,d,T}>0$ such that
$$\sup_{\varepsilon>0} E[\|I_3^{\varepsilon}(\theta',\theta)\|^2] \leq C |\theta|^{2\left(H-\frac{1}{2}\right)}|\theta - \theta'|^{\delta}$$
provided $H<\frac{1}{2(2+d)}$. Then again, it is easy to see that we can choose $\beta\in(0,1/2)$ small enough so that
it follows from the proof of Lemma \ref{doubleint} that
\begin{align}\label{double3}
\int_0^t \int_0^t |\theta|^{2\left(H-\frac{1}{2}\right)}|\theta - \theta'|^{\varepsilon -1 -2\beta} d\theta' d\theta <\infty,
\end{align}
for every $t\in [0,T]$.

Altogether, taking a suitable $\beta$ so that \eqref{double1}, \eqref{double2} and \eqref{double3} are finite, we have
$$\sup_{\varepsilon> 0} \int_0^t \int_0^t \frac{E[\| D_{\theta'} X_t^{\varepsilon} - D_{\theta} X_t^{\varepsilon} \|^2]}{|\theta' - \theta|^{1+2\beta}} d\theta' d\theta <\infty.$$

Similar computations show that
$$\sup_{\varepsilon>0} \|D_{\cdot} X_t^{\varepsilon}\|_{L^2(\Omega\times [0,T],\R^{d\times d})} < \infty.$$
\end{proof}

\begin{cor}\label{L2conv}
Let $\{X_t^{\varepsilon}\}_{\varepsilon>0}$ the family of approximating solutions of \eqref{SDE} in the sense of \eqref{Xn}. Then for every $t\in [0,T]$ and bounded continuous function $h:\R^d \rightarrow \R$ we have
$$h(X_t^{n}) \xrightarrow{n \to \infty} h(E\left[ X_t |\mathcal{F}_t \right])$$
strongly in $L^2(\Omega; \mathcal{F}_t)$. In addition, $E\left[X_t|\mathcal{F}_t\right]$ is Malliavin differentiable for every $t\in [0,T]$.
\end{cor}
\begin{proof}
This is an immediate consequence of the relative compactness from Lemma \ref{relcomp} and by Lemma \ref{weakconv} we can identify the limit as being $E[X_t|\mathcal{F}_t]$ then the convergence holds for any bounded continuous functions as well. The Malliavin differentiability of $E[X_t|\mathcal{F}_t]$ is shown by taking $h =I_d$ and estimate \eqref{Mallest} together with \cite[Proposition 1.2.3]{Nua10}.
\end{proof}

Finally, we can prove the main result of this section.

\begin{proof}[Proof of Theorem \ref{mainthm}]
It remains to prove that $X_t$ is $\mathcal{F}_t$-measurable for every $t\in [0,T]$. It follows that there exists a strong solution in the usual sense that is Malliavin differentiable. Indeed, let $h$ be a globally Lipschitz continuous function, then by Corollary \ref{L2conv} we have that
$$\varphi(X_t^{1/n}) \rightarrow \varphi(E[X_t|\mathcal{F}_t]), \ \ P-a.s.$$
as $n\to \infty$.

On the other hand, by Lemma \ref{weakconv} we also have
$$h (X_t^{1/n}) \rightarrow E\left[ \varphi(X_t)|\mathcal{F}_t\right]$$
weakly in $L^2(\Omega;\mathcal{F}_t)$ as $n\to \infty$. By the uniqueness of the limit we immediately have
$$h\left( E[X_t|\mathcal{F}_t] \right) = E\left[ h(X_t)|\mathcal{F}_t\right], \ \ P-a.s.$$
which implies that $X_t$ is $\mathcal{F}_t$-measurable for every $t\in [0,T]$.

Let us finally show that our strong solution has a continuous modification.
We observe that
\begin{eqnarray*}
&&E[\left\vert X_{t}^{x}-X_{s}^{x}\right\vert ^{m}] \\
&\leq &C_{d,m}(E[(\int_{s}^{t}\delta _{0}(X_{u}^{x})du)^{m}]+E[\left\vert
B_{t}^{H}-B_{s}^{H}\right\vert ^{m}]) \\
&\leq &C_{d,m}(E[(\int_{s}^{t}\delta _{0}(X_{u}^{x})du)^{m}]+\left\vert
t-s\right\vert ^{mH}).
\end{eqnarray*}%
On the other hand, we have that%
\begin{equation*}
E[(\int_{s}^{t}\delta _{0}(X_{u}^{x})du)^{m}]\leq E[(\int_{s}^{t}\delta
_{0}(B_{u}^{H}+x)du)^{2m}]^{1/2}E[X^{2}]^{1/2},
\end{equation*}%
where $X$ is the Radon-Nikodym derivative as constructed in Proposition \ref{weaksolution}.
Further, we know from (\ref{eq:BoundMomentLT}) for a similar estimate that
\begin{equation*}
E[(\int_{s}^{t}\delta _{0}(B_{u}^{H}+x)du)^{2m}]^{1/2}\leq
C_{d,m,H}\left\vert t-s\right\vert ^{\frac{m}{2}(1-Hd)}
\end{equation*}%
So 
\begin{equation*}
E[\left\vert X_{t}^{x}-X_{s}^{x}\right\vert ^{m}]\leq C(\left\vert
t-s\right\vert ^{\frac{m}{2}(1-Hd)}+\left\vert t-s\right\vert ^{mH}),s\leq
t,m\geq 1,
\end{equation*}%
which entails by Kolmogorov's Lemma the existence of a continuous
modification of $X_{\cdot }^{x}$.

\end{proof}

\begin{proof}[Proof of Proposition \ref{mainprop}]
Denote by $Y$ the $L^{p}-$limit of the Doleans-Dade
exponentials. Using characteristic functions combined with Novikov's
condition, we see that%
\begin{equation*}
Y_{t}^{x}-x=B_{t}^{H}+L_{t}(Y^{x})\boldsymbol{1}_{d}
\end{equation*}%
is a fractional Brownian motion under a change of measure with respect to
the density $Y$. The latter enables us to proceed similarly to arguments in
the proof of Lemma \ref{weakconv} and to verify that%
\begin{equation*}
E[Y_{t}^{x}\exp \{\sum_{j=1}^{k}\left\langle \alpha
_{j},B_{t_{j}}^{H}-B_{t_{j-1}}^{H}\right\rangle \}]=E[X_{t}^{x}\exp
\{\sum_{j=1}^{k}\left\langle \alpha
_{j},B_{t_{j}}^{H}-B_{t_{j-1}}^{H}\right\rangle \}]
\end{equation*}%
for all $\{\alpha _{j}\}_{j=1}^{k}\subset \mathbb{R}%
^{d},0=t_{0}<...<t_{k}=t,k\geq 1$, where $X_{\cdot }^{x}$ denotes the
constructed strong solution of our main theorem. This allows us conclude
that both solutions must coincide a.e.
\end{proof}

\appendix

\section{Technical results}
The following result which is due to \cite[Theorem 1] {DPMN92} provides a compactness criterion for subsets of $L^{2}(\Omega)$ using Malliavin calculus.

\begin{thm}
\label{MCompactness}Let $\left\{ \left( \Omega ,\mathcal{A},P\right)
;H\right\} $ be a Gaussian probability space, that is $\left( \Omega ,%
\mathcal{A},P\right) $ is a probability space and $H$ a separable closed
subspace of Gaussian random variables of $L^{2}(\Omega )$, which generate
the $\sigma $-field $\mathcal{A}$. Denote by $\mathbf{D}$ the derivative
operator acting on elementary smooth random variables in the sense that%
\begin{equation*}
\mathbf{D}(f(h_{1},\ldots,h_{n}))=\sum_{i=1}^{n}\partial
_{i}f(h_{1},\ldots,h_{n})h_{i},\text{ }h_{i}\in H,f\in C_{b}^{\infty }(\mathbb{R%
}^{n}).
\end{equation*}%
Further let $\mathbb{D}^{1,2}$ be the closure of the family of elementary
smooth random variables with respect to the norm%
\begin{align*}
\left\Vert F\right\Vert _{1,2}:=\left\Vert F\right\Vert _{L^{2}(\Omega
)}+\left\Vert \mathbf{D}F\right\Vert _{L^{2}(\Omega ;H)}.
\end{align*}%
Assume that $C$ is a self-adjoint compact operator on $H$ with dense image.
Then for any $c>0$ the set
\begin{equation*}
\mathcal{G}=\left\{ G\in \mathbb{D}^{1,2}:\left\Vert G\right\Vert
_{L^{2}(\Omega )}+\left\Vert C^{-1} \mathbf{D} \,G\right\Vert _{L^{2}(\Omega ;H)}\leq
c\right\}
\end{equation*}%
is relatively compact in $L^{2}(\Omega )$.
\end{thm}

In order to formulate compactness criteria useful for our purposes, we need the following technical result which also can be found in \cite{DPMN92}.

\begin{lem}
\label{DaPMN} Let $v_{s},s\geq 0$ be the Haar basis of $L^{2}([0,T])$. For
any $0<\alpha <1/2$ define the operator $A_{\alpha }$ on $L^{2}([0,T])$ by%
\begin{equation*}
A_{\alpha }v_{s}=2^{k\alpha }v_{s}\text{, if }s=2^{k}+j\text{ }
\end{equation*}%
for $k\geq 0,0\leq j\leq 2^{k}$ and%
\begin{equation*}
A_{\alpha }1=1.
\end{equation*}%
Then for all $\beta $ with $\alpha <\beta <(1/2),$ there exists a constant $%
c_{1}$ such that%
\begin{equation*}
\left\Vert A_{\alpha }f\right\Vert \leq c_{1}\left\{ \left\Vert f\right\Vert
_{L^{2}([0,T])}+\left(\int_{0}^{T}\int_{0}^{T}\frac{\left|
f(t)-f(t^{\prime })\right|^2}{\left\vert t-t^{\prime }\right\vert
^{1+2\beta }}dt\,dt^{\prime }\right)^{1/2}\right\} .
\end{equation*}
\end{lem}

A direct consequence of Theorem \ref{MCompactness} and Lemma \ref{DaPMN} is now the following compactness criteria.

\begin{cor} \label{compactcrit}
Let a sequence of $\mathcal{F}_T$-measurable random variables $X_n\in\mathbb{D}^{1,2}$, $n=1,2...$, be such that there exists a constant $C>0$ with
$$
\sup_n E[|X_n|^2] \leq C ,
$$
$$
\sup_n E \left[ \| D_t X_n \|_{L^2([0,T])}^2 \right] \leq C \,
$$
and there exists a $\beta \in (0,1/2)$ such that
$$
\sup_n \int_0^T \int_0^T \frac{E \left[ \| D_t X_n - D_{t'} X_n \|^2 \right]}{|t-t'|^{1+2\beta}} dtdt' <\infty
$$
where $\|\cdot\|$ denotes any matrix norm.

Then the sequence $X_n$, $n=1,2...$, is relatively compact in $L^{2}(\Omega )$.
\end{cor}

For the use of the above result we will need to exploit the following technical results which are extracted from \cite{BNP.16}.

\begin{lem}\label{doubleint}
Let $H \in (0,1/2)$ and $t\in [0,T]$ be fixed. Then, there exists a $\beta \in (0,1/2)$ such that
\begin{align}\label{intI}
\int_0^t \int_0^t \frac{|K_H(t,\theta') - K_H(t,\theta)|^2}{|\theta'-\theta|^{1+2\beta}}d\theta d\theta ' < \infty.
\end{align}
\end{lem}
\begin{proof}
Let $\theta,\theta'\in [0,t]$, $\theta'<\theta$ be fixed. Write
$$K_H (t,\theta) - K_H(t,\theta') = c_H\left[f_t(\theta) - f_t(\theta') + \left(\frac{1}{2}-H\right) \left(g_t(\theta) - g_t(\theta')\right)\right],$$
where $f_t (\theta):= \left(\frac{t}{\theta} \right)^{H-\frac{1}{2}} (t-\theta)^{H-\frac{1}{2}}$ and $g_t(\theta) := \int_\theta^t \frac{f_u (\theta)}{u}du$, $\theta\in [0,t]$.

We will proceed to estimating $K_H (t,\theta) - K_H(t,\theta')$. First, observe the following fact,
$$\frac{y^{-\alpha} -x^{-\alpha}}{(x-y)^{\gamma}} \leq C y^{-\alpha-\gamma}$$
for every $0<y<x<\infty$ and $\alpha :=(\frac{1}{2}-H) \in (0,1/2)$ and $\gamma < \frac{1}{2}-\alpha$. This implies
\begin{align*}
f_t(\theta) - f_t(\theta') &= \left(\frac{t}{\theta}  (t-\theta)\right)^{H-\frac{1}{2}}-\left(\frac{t}{\theta'} (t-\theta')\right)^{H-\frac{1}{2}}\\
&\leq C \left(\frac{t}{\theta}(t-\theta )\right)^{H-\frac{1}{2} -\gamma }t^{2\gamma }\frac{(\theta -\theta')^{\gamma }}{(\theta \theta')^{\gamma }}\\
&\leq C\frac{(\theta -\theta')^{\gamma }}{(\theta \theta')^{\gamma }}(t-\theta )^{H-\frac{1}{2}-\gamma } \\
&\leq C\frac{(\theta -\theta')^{\gamma }}{(\theta \theta
')^{\gamma }}\theta^{H-\frac{1}{2}-\gamma }(t-\theta )^{H-\frac{1}{2}-\gamma }.
\end{align*}

Further, 
\begin{align*}
g_{t}(\theta )-g_{t}(\theta') &= \int_{\theta }^{t}\frac{f_{u}(\theta )-f_{u}(\theta')}{u}du-\int_{\theta'}^{\theta }\frac{f_{u}(\theta')}{u}du \\
&\leq \int_{\theta }^{t}\frac{f_{u}(\theta )-f_{u}(\theta')}{u}du \\
&\leq C\frac{(\theta -\theta')^{\gamma }}{(\theta \theta')^{\gamma }}\int_{\theta }^{t}\frac{(u-\theta )^{H-\frac{1}{2}-\gamma }}{u}du \\
&\leq C\frac{(\theta -\theta')^{\gamma }}{(\theta \theta')^{\gamma }}\theta^{H-\frac{1}{2}-\gamma }\int_{1}^{\infty }\frac{(u-1)^{H-\frac{1}{2}-\gamma }}{u}du \\
&\leq C\frac{(\theta -\theta')^{\gamma }}{(\theta \theta')^{\gamma }}\theta^{H-\frac{1}{2}-\gamma } \\
&\leq C\frac{(\theta -\theta')^{\gamma }}{(\theta \theta')^{\gamma }}\theta^{H-\frac{1}{2}-\gamma }(t-\theta )^{H-\frac{1}{2}-\gamma }.
\end{align*}

As a result, we have for every $\gamma\in (0,H)$, $0<\theta'<\theta<t<T$,
\begin{align*}
(K_{H}(t,\theta )-K_{H}(t,\theta'))^{2}\leq C_{H,T}\frac{(\theta
-\theta')^{2\gamma }}{(\theta \theta')^{2\gamma }}\theta^{2H-1-2\gamma }(t-\theta )^{2H-1-2\gamma },
\end{align*}
for some constant $C_{H,T}>0$ depending only on $H$ and $T$.

Thus
\begin{align*}
\int_{0}^{t}\int_{0}^{\theta }&\frac{(K_{H}(t,\theta )-K_{H}(t,\theta'))^{2}}{| \theta -\theta|^{1+2\beta }}d\theta'd\theta  \\
&\leq C\int_{0}^{t}\int_{0}^{\theta }\frac{|\theta -\theta'|^{-1-2\beta +2\gamma }}{(\theta \theta')^{2\gamma }}\theta^{2H-1-2\gamma }(t-\theta )^{2H-1-2\gamma }d\theta'd\theta  \\
& =C\int_{0}^{t}\theta^{2H-1-4\gamma }(t-\theta )^{2H-1-2\gamma
}\int_{0}^{\theta }|\theta -\theta'|^{-1-2\beta +2\gamma }(\theta')^{-2\gamma }d\theta'd\theta  \\
&= C\int_{0}^{t}\theta^{2H-1-4\gamma }(t-\theta )^{2H-1-2\gamma }\frac{\Gamma (-2\beta +2\gamma )\Gamma (-2\gamma +1)}{\Gamma (-2\beta +1)}\theta
^{-2\beta }d\theta  \\
&\leq C\int_{0}^{t}\theta^{2H-1-4\gamma -2\beta }(t-\theta
)^{2H-1-2\gamma }d\theta  \\
&=C\frac{\Gamma (2H-2\gamma )\Gamma (2H-4\gamma -2\beta )}{\Gamma
(4H-6\gamma -2\beta )}t^{4H-6\gamma -2\beta -1}<\infty,
\end{align*}%
for appropriately chosen small $\gamma $ and $\beta$.

On the other hand, we have that
\begin{align*}
\int_{0}^{t}\int_{\theta }^{t}&\frac{(K_{H}(t,\theta )-K_{H}(t,\theta'))^{2}}{|\theta -\theta'|^{1+2\beta }}d\theta' d\theta  \\
&\leq C\int_{0}^{t}\theta^{2H-1-4\gamma }(t-\theta )^{2H-1-2\gamma
}\int_{\theta }^{t}\frac{|\theta -\theta'|^{-1-2\beta +2\gamma }}{(\theta')^{2\gamma }}d\theta'd\theta  \\
&\leq C\int_{0}^{t}\theta^{2H-1-6\gamma }(t-\theta )^{2H-1-2\gamma
} \int_{\theta}^t |\theta -\theta'|^{-1-2\beta +2\gamma } d\theta' d\theta  \\
&=C\int_{0}^{t}\theta^{2H-1-6\gamma }(t-\theta )^{2H-1 -2\beta
}d\theta  \\
&\leq Ct^{4H-6\gamma -2\beta -1}.
\end{align*}%
Hence
\begin{align*}
\int_{0}^{t}\int_{0}^{t}\frac{(K_{H}(t,\theta )-K_{H}(t,\theta'))^{2}}{|\theta -\theta'|^{1+2\beta }}%
d\theta' d\theta <\infty .
\end{align*}
\end{proof}

\begin{lem}
\label{lem:Integral}If $H<\frac{1}{2\left(1+d\right)}$ we have that
\begin{align*}
\mathcal{I} & :=\left(2m\right)!\int_{\mathcal{T}_{2m}\left(0,1\right)}\prod_{j=1}^{2m}\gamma_{-\frac{1}{2}-H,\frac{1}{2}-H}\left(1,u_{j}\right)\left(\det\mathrm{Cov}\left(B_{u_{1}}^{H,1},\ldots,B_{u_{2m}}^{H,1}\right)\right)^{-\frac{d}{2}}d\mathbf{u}\leq C_{H,d}^{m}\left(m!\right)^{2H\left(1+d\right)},
\end{align*}
for some constant $C_{H,d}$ depending only on $H$ and $d.$\end{lem}
\begin{proof}
We have that
\begin{align*}
\mathcal{I} & =\left(2m\right)!\int_{\mathcal{T}_{2m}\left(0,1\right)}\prod_{j=1}^{2m}\gamma_{-\frac{1}{2}-H,\frac{1}{2}-H}\left(1,u_{j}\right)\left(\det\mathrm{Cov}\left(B_{u_{1}}^{H,1},\ldots,B_{u_{2m}}^{H,1}\right)\right)^{-\frac{d}{2}}d\mathbf{u}\\
 & \leq\left(2m\right)!\int_{\mathcal{T}_{2m}\left(0,1\right)}\gamma_{-\frac{1}{2}-H,\frac{1}{2}-H}\left(1,u_{1}\right)u_{1}^{-Hd}\prod_{j=2}^{2m}\gamma_{-\frac{1}{2}-H,\frac{1}{2}-H}\left(1,u_{j}\right)\left(u_{j}-u_{j-1}\right)^{-Hd}d\mathbf{u}\\
 & \leq\left(2m\right)!\int_{0}^{1}\int_{0}^{u_{2m}}\cdots\int_{0}^{u_{3}}\left(\gamma_{H}\left(1,u_{2}\right)\prod_{j=3}^{2m}\gamma_{-\frac{1}{2}-H,\frac{1}{2}-H}\left(1,u_{j}\right)\left(u_{j}-u_{j-1}\right)^{-Hd}\right)\\
 & \qquad\times\left(\int_{0}^{u_{2}}\gamma_{-\frac{1}{2}-H,\frac{1}{2}-H}\left(1,u_{1}\right)u_{1}^{-Hd}\left(u_{2}-u_{1}\right)^{-Hd}du_{1}\right)du_{2}\cdots du_{2m}.
\end{align*}
The inner integral can be bounded by
\begin{align*}
 & \int_{0}^{u_{2}}\gamma_{-\frac{1}{2}-H,\frac{1}{2}-H}\left(1,u_{1}\right)u_{1}^{-Hd}\left(u_{2}-u_{1}\right)^{-Hd}du_{1}\\
 & =u_{2}^{\frac{3}{2}-2Hd-H}\int_{0}^{1}\gamma_{-Hd,\frac{1}{2}-H\left(1+d\right)}\left(1,u_{1}\right)\left(1-u_{2}u_{1}\right)^{-\frac{1}{2}-H}du_{1}\\
 & \leq u_{2}^{\frac{3}{2}-2Hd-H}\int_{0}^{1}\gamma_{-\frac{1}{2}-H\left(1+d\right),\frac{1}{2}-H\left(1+d\right)}\left(1,u_{1}\right)du_{1}\\
 & =u_{2}^{\frac{3}{2}-2Hd-H}\mathcal{B}\left(\frac{1}{2}-H\left(1+d\right),\frac{3}{2}-H\left(1+d\right)\right),
\end{align*}
where we have used that $\left(1-u_{2}u_{1}\right)^{-\frac{1}{2}-H}\leq\left(1-u_{1}\right)^{-\frac{1}{2}-H}$. Hence,
\begin{align*}
\mathcal{I} & \leq\left(2m\right)!\int_{0}^{1}\int_{0}^{u_{2m}}\cdots\int_{0}^{u_{4}}\left(\gamma_{-\frac{1}{2}-H,\frac{1}{2}-H}\left(1,u_{3}\right)\prod_{j=4}^{2m}\gamma_{-\frac{1}{2}-H,\frac{1}{2}-H}\left(1,u_{j}\right)\left(u_{j}-u_{j-1}\right)^{-Hd}\right)\\
 & \qquad\times\left(\int_{0}^{u_{3}}\gamma_{-\frac{1}{2}-H,2-2Hd-2H}\left(1,u_{2}\right)\left(u_{3}-u_{2}\right)^{-Hd}du_{2}\right)du_{3}\cdots du_{2m}\\
 & \qquad\times\mathcal{B}\left(\frac{1}{2}-H\left(1+d\right),\frac{3}{2}-H\left(1+d\right)\right).
\end{align*}
The inner integral can be bounded by
\begin{align*}
 & \int_{0}^{u_{3}}\gamma_{-\frac{1}{2}-H,2-2Hd-2H}\left(1,u_{2}\right)\left(u_{3}-u_{2}\right)^{-Hd}du_{2}\\
 & =u_{3}^{3-3Hd-2H}\int_{0}^{1}\gamma_{-Hd,2-2Hd-2H}\left(1,u_{2}\right)\left(1-u_{3}u_{2}\right)^{-\frac{1}{2}-H}du_{2}\\
 & \leq u_{3}^{3-3Hd-2H}\int_{0}^{1}\gamma_{-\frac{1}{2}-H\left(1+d\right),2-2Hd-2H}\left(1,u_{2}\right)du_{2}\\
 & =u_{3}^{3-3Hd-2H}\mathcal{B}\left(\frac{1}{2}-H\left(1+d\right),3-2H\left(1+d\right)\right),
\end{align*}
and we get
\begin{align*}
\mathcal{I} & \leq\left(2m\right)!\int_{0}^{1}\int_{0}^{u_{2m}}\cdots\int_{0}^{u_{5}}\left(\gamma_{-\frac{1}{2}-H,\frac{1}{2}-H}\left(1,u_{4}\right)\prod_{j=5}^{2m}\gamma_{-\frac{1}{2}-H,\frac{1}{2}-H}\left(1,u_{j}\right)\left(u_{j}-u_{j-1}\right)^{-Hd}\right)\\
 & \qquad\times\left(\int_{0}^{u_{4}}\gamma_{-\frac{1}{2}-H,\frac{7}{2}-3H\left(1+d\right)}\left(1,u_{3}\right)\left(u_{4}-u_{3}\right)^{-Hd}du_{3}\right)du_{4}\cdots du_{2m}\\
 & \qquad\times\mathcal{B}\left(\frac{1}{2}-H\left(1+d\right),\frac{3}{2}-H\left(1+d\right)\right)\mathcal{B}\left(\frac{1}{2}-H\left(1+d\right),3-2H\left(1+d\right)\right).
\end{align*}
Iterating the previous reasoning we have that
\begin{align*}
\mathcal{I} & \leq\left(2m\right)!\prod_{j=1}^{2m}\mathcal{B}\left(\frac{1}{2}-H\left(1+d\right),j\left(\frac{3}{2}-H\left(1+d\right)\right)\right)\\
 & =\left(2m\right)!\prod_{j=1}^{2m}\frac{\Gamma\left(\frac{1}{2}-H\left(1+d\right)\right)\Gamma\left(j\left(\frac{3}{2}-H\left(1+d\right)\right)\right)}{\Gamma\left(\frac{1}{2}-H\left(1+d\right)+j\left(\frac{3}{2}-H\left(1+d\right)\right)\right)}\\
 & =\left(\Gamma\left(\frac{1}{2}-H\left(1+d\right)\right)\right)^{2m}\frac{\Gamma\left(\frac{3}{2}-H\left(1+d\right)\right)\left(2m\right)!}{\Gamma\left(\frac{1}{2}-H\left(1+d\right)+2m\left(\frac{3}{2}-H\left(1+d\right)\right)\right)}\\
 & \qquad\times\prod_{j=1}^{2m-1}\frac{\Gamma\left(1+\frac{1}{2}-H\left(1+d\right)+j\left(\frac{3}{2}-H\left(1+d\right)\right)\right)}{\Gamma\left(\frac{1}{2}-H\left(1+d\right)+j\left(\frac{3}{2}-H\left(1+d\right)\right)\right)}\\
 & =\left(\Gamma\left(\frac{1}{2}-H\left(1+d\right)\right)\right)^{2m}\frac{\Gamma\left(\frac{3}{2}-H\left(1+d\right)\right)\left(2m\right)!}{\Gamma\left(\frac{1}{2}-H\left(1+d\right)+2m\left(\frac{3}{2}-H\left(1+d\right)\right)\right)}\\
 & \qquad\times\prod_{j=1}^{2m-1}\left(\frac{1}{2}-H\left(1+d\right)+j\left(\frac{3}{2}-H\left(1+d\right)\right)\right)\\
 & =\left(\Gamma\left(\frac{1}{2}-H\left(1+d\right)\right)\right)^{2m}\frac{\Gamma\left(\frac{3}{2}-H\left(1+d\right)\right)\Gamma\left(2m+1\right)}{\Gamma\left(-\left(\frac{1}{2}+H\left(1+d\right)\right)+m\left(3-2H\left(1+d\right)\right)+1\right)}\\
 & \qquad\times\left(\frac{3}{2}-H\left(1+d\right)\right)^{2m-1}\frac{\Gamma\left(-\frac{2}{3-2H\left(1+d\right)}+2m+1\right)}{\Gamma\left(\frac{4\left(1-H(1+d)\right)}{3-2H\left(1+d\right)}\right)}.
\end{align*}
Next, taking into account the following asymptotics, see Wendel \cite{W48},
\begin{align*}
\Gamma\left(m+\lambda\right) & \sim m^{\lambda}\Gamma\left(m\right),\\
\Gamma\left(\lambda m+1\right) & \sim\lambda^{\frac{1}{2}}\left(2\pi\right)^{\frac{1-\lambda}{2}}\lambda^{\lambda m}m^{\frac{1-\lambda}{2}}\left(m!\right)^{\lambda},
\end{align*}
we get that
\begin{align*}
\Gamma\left(2m+1\right) & \sim2^{\frac{1}{2}}\left(2\pi\right)^{-\frac{1}{2}}4^{m}m^{-\frac{1}{2}}\left(m!\right)^{2},\\
\Gamma\left(-\left(\frac{1}{2}+H\left(1+d\right)\right)+m\left(3-2H\left(1+d\right)\right)+1\right) & \sim\left(m\left(3-2H\left(1+d\right)\right)+1\right)^{-\left(\frac{1}{2}+H(1+d)\right)}\\
 & \qquad\times\Gamma\left(m\left(3-2H\left(1+d\right)\right)+1\right)\\
 & \sim C_{H,d}K_{H,d}^{m}\left(m\right)^{-\frac{3}{2}}\left(m!\right)^{3-2H(1+d)},\\
\Gamma\left(-\frac{2}{3-2H\left(1+d\right)}+2m+1\right) & \sim\left(2m+1\right)^{-\frac{2}{3-2H\left(1+d\right)}}\Gamma\left(2m+1\right)\\
 & \sim C'_{H,d}\left(K'_{H,d}\right)^{m}\left(m\right)^{-\frac{2}{3-2H\left(1+d\right)}-\frac{1}{2}}\left(m!\right)^{2}
\end{align*}
which yields
\[
\frac{\Gamma\left(2m+1\right)\Gamma\left(-\frac{2}{3-2H\left(1+d\right)}+2m+1\right)}{\Gamma\left(-\left(\frac{1}{2}+H\left(1+d\right)\right)+m\left(3-2H\left(1+d\right)\right)+1\right)}\sim C''_{H,d}\left(K''_{H,d}\right)^{m}\left(m\right)^{-\frac{1+2H(1+d)}{6-4H\left(1+d\right)}}\left(m!\right)^{2H\left(1+d\right)},
\]
and the result follows.
\end{proof}

The next auxiliary result can be found in \cite{LiWei}. 

\begin{lem}\label{LiWei}
Assume that $X_{1},...,X_{n}$ are real centered jointly Gaussian
random variables, and $\Sigma =(E[XjX_{k}])_{1\leq j,k\leq n}$ is the
covariance matrix, then%
\begin{equation*}
E[\left\vert X_{1}\right\vert ...\left\vert X_{n}\right\vert ]\leq \sqrt{%
perm(\Sigma )},
\end{equation*}%
where $perm(A)$ is the permanent of a matrix $A=(a_{ij})_{1\leq i,j\leq n}$
defined by%
\begin{equation*}
perm(A)=\sum_{\pi \in S_{n}}\prod_{j=1}^{n}a _{j,\pi (j)}
\end{equation*}%
for the symmetric group $S_{n}$.

\end{lem}

The next result corresponds to Lemma 3.19 in \cite{CD}:
\begin{lem}\label{CD}
Let $Z_{1},...,Z_{n}$ be mean zero Gaussian variables which are linearly
independent. Then for any measurable function $g:\mathbb{R}\longrightarrow 
\mathbb{R}_{+}$ we have that%
\begin{equation*}
\int_{\mathbb{R}^{n}}g(v_{1})\exp (-\frac{1}{2}Var[%
\sum_{j=1}^{n}v_{j}Z_{j}])dv_{1}...dv_{n}=\frac{(2\pi )^{(n-1)/2}}{(\det
Cov(Z_{1},...,Z_{n}))^{1/2}}\int_{\mathbb{R}}g(\frac{v}{\sigma _{1}})\exp (-%
\frac{1}{2}v^{2})dv,
\end{equation*}%
where $\sigma _{1}^{2}:=Var[Z_{1}\left\vert Z_{2},...,Z_{n}\right] $.

\end{lem}
The following Lemma is Lemma A.5 in \cite{BNP.16}: 
\begin{lem}\label{VI_iterativeInt}
Let $H \in (0,1/2)$, $\theta,t\in [0,T]$, $\theta<t$ and $(\varepsilon_1,\dots, \varepsilon_{m})\in \{0,1\}^{m}$ be fixed. Assume $w_j+\left(H-\frac{1}{2}-\gamma\right) \varepsilon_j>-1$ for all $j=1,\dots,m$. Then exists a finite constant $C=C(H,T)>0$ such that
\begin{align*}
\int_{\Delta_{\theta,t}^{m}}  &\prod_{j=1}^{m} (K_H(s_j,\theta) - K_H(s_j,\theta'))^{\varepsilon_j} |s_j-s_{j-1}|^{w_j} ds\\
\leq& C^m \left(\frac{\theta-\theta'}{\theta \theta'}\right)^{\gamma \sum_{j=1}^m \varepsilon_j} \theta^{\left( H-\frac{1}{2}-\gamma\right)\sum_{j=1}^m \varepsilon_j} \,  \Pi_{\gamma}(m) \, (t-\theta)^{\sum_{j=1}^m w_j + \left( H-\frac{1}{2}-\gamma\right) \sum_{j=1}^m \varepsilon_j +m}
\end{align*}
for $\gamma \in (0,H)$, where
\begin{align}\label{VI_Pi}
\Pi_{\gamma}(m):=\prod_{j=1}^{m-1} \frac{\Gamma \left(\sum_{l=1}^{j} w_l + \left(H-\frac{1}{2}-\gamma \right)\sum_{l=1}^{j} \varepsilon_l + j\right)\Gamma \left( w_{j+1}+1\right)}{\Gamma \left( \sum_{l=1}^{j+1} w_l + \left(H-\frac{1}{2}-\gamma \right)\sum_{l=1}^{j} \varepsilon_l + j+1\right)}.
\end{align}
Observe that if $\varepsilon_j=0$ for all $j=1,\dots, m$ we obtain the classical formula.
\end{lem}

\end{document}